\documentclass[a4paper,11pt]{article}
\usepackage{amsmath}                    
\usepackage{amssymb}                    
\usepackage{amsthm}                     
\usepackage{amsfonts}                   
\usepackage{subcaption}
\usepackage[usenames, dvipsnames]{color}
\usepackage{graphics,graphicx}
\usepackage{enumerate}
\usepackage{dsfont}
\usepackage{float}
\usepackage{tikz}
\usetikzlibrary{matrix}
\usepackage{hyperref}

\allowdisplaybreaks

\newtheorem{theorem}{Theorem}[section]

\newtheorem{proposition}[theorem]{Proposition}
\newtheorem{lemma}[theorem]{Lemma}
\newtheorem{remark}[theorem]{Remark}

\newtheorem{algo}[theorem]{Algorithm}
\def\ts{\thinspace}

\DeclareMathOperator{\diag}{diag}

\renewcommand{\epsilon}{\varepsilon}
\newcommand{\eps}{\varepsilon}

\newcommand{\R}{\mathbb{R}}
\newcommand{\bigo}{\mathcal{O}}

\def \IR{\mathbb R}

\def \N{\mathbb N}

\def \eps{\epsilon}

\def \L{\mathcal{L}}

\def \C{\mathcal{C}}

\renewcommand{\R}{\mathbb{R}}
\newcommand{\dif}[2]{\frac{\partial #1}{\partial #2}}

\title{
	Explicit stabilized integrators for stiff optimal control problems
}

\author{
	Ibrahim Almuslimani\textsuperscript{1} and Gilles Vilmart\textsuperscript{1}
}

\begin{document}
	\maketitle
	\footnotetext[1]{Universit\'e de Gen\`eve, Section de math\'ematiques, 2-4 rue du Li\`evre, CP 64, CH-1211 Gen\`eve 4, Switzerland, Ibrahim.Almuslimani@unige.ch, Gilles.Vilmart@unige.ch}
	\begin{abstract}
				Explicit stabilized methods are an efficient alternative to implicit schemes for the time integration of stiff systems of differential equations in large dimension. In this paper we derive explicit stabilized integrators of orders one and two for the optimal control of stiff systems. We analyze their favorable stability properties based on the continuous optimality conditions. { Furthermore, we study their order of convergence taking advantage of the symplecticity of the corresponding partitioned Runge-Kutta method involved for the adjoint equations.}
			Numerical experiments including the optimal control of a nonlinear diffusion-advection PDE illustrate the efficiency of the new approach.
		
		\smallskip
		
		
	\end{abstract}

	
\section{Introduction}\label{sec:intro}
In this paper, we introduce and analyze numerical methods for the optimal control of systems of ordinary differential equations (ODEs) of the form
\begin{equation}\label{eq:P} 
	\min_u ~\Psi(y(T));\qquad \dot{y}(t):=\frac{dy}{dt}(t)=f(u(t),y(t)),\quad t\in [0,T];~~y(0)=y^0,
\end{equation}
where for a fixed final time $T>0$ and a given initial condition $y^0\in \R^n$, the function $y:[0,T]\to \IR^n$ is the unknown state function, $u:[0,T]\to \IR^m$ is the unknown control function.
Here, $f:\R^m \times \R^n \to \R^n$ is the given vector field and 
$\Psi:\R^n \to \R$ is the given cost function, which are assumed to be $C^\infty$ mappings. 
For simplicity of the presentation, we consider the case of autonomous problems (with $f$ independent of time) but we highlight that our approach also applies straightforwardly to non-autonomous problems $\frac{dy}{dt}(t)=f(t,u(t),y(t))$.\footnote[2]{A standard approach is to consider the augmented system with $z(t)=t$, i.e. $\frac{dz}{dt}=1,~z(0)=0$ and define $\tilde{y}(t)=(y(t),z(t))^T$, see e.g. \cite[Chap.\ts III]{HLW06} for details.}

There are essentially two approaches for the numerical solution of optimal control problems: the direct approach, which consists in directly discretizing \eqref{eq:P} and then applying a minimization method to the corresponding discrete minimization problem, and the indirect approach, which is based on Pontryagin's maximum principle, taking benefit of continuous optimality conditions (adjoint equation).
A natural approach for the accurate numerical approximation of such optimal control problem \eqref{eq:P} is to consider Runge-Kutta type schemes. It was shown in \cite[Theorem\ts4.1]{Hager00} by studying the continuous and discrete optimality conditions that additional order conditions for the convergence rate are required in general by Runge-Kutta methods when applied to optimal control problems, compared to the integration of standard initial valued ordinary differential equations and conditions up to order $4$ were derived. In \cite{BL06}, general order conditions were derived, in addition to identifying symplecticity properties. This result is related to the order of symplectic partitioned Runge-Kutta methods, and it implies in particular that applying naively a Runge-Kutta method to \eqref{eq:P} yields  in general an order reduction phenomenon.
{  This analysis was then extended to other classes of Runge-Kutta type schemes in \cite{Kay10, LaV13, HPS13}, see also \cite{HeS11,AHP19} in the context of hyperbolic problems and multistep methods.} The use of symplectic integrators is motivated by the recent publication \cite{LF19} which proves the convergence of  forward-backward iterative algorithm { (Algorithm \ref{algo:iter} in the present paper)}, to implement discretized optimal control problems, when using a symplectic Runge-Kutta method. The work done in \cite{LF19} generalizes that of \cite{LCTE18} in which the authors prove the global convergence of the algorithm in the continuous time case. We also mention the paper \cite{Wal07} where automatic differentiation can be efficiently applied for computing the gradient of the cost function {  under the assumption that optimal control order conditions are satisfied.} In our algorithms, the Jacobian of the vector field is given as an input, however the idea of automatic differentiation could be coupled with our approach to compute derivatives automatically, {  but this is not the purpose of the present paper.}

In the case where the vector field $f$ in \eqref{eq:P} is stiff, due for instance to the multiscale nature of the model, or due to the spatial discretization of a diffusion operator in a partial differential equation (PDE) model, standard explicit integrators face in general a severe time step restriction making standard explicit methods unreasonable to be used due to their dramatic cost. A standard approach in this stiff case is to consider indirect implicit methods with good stability properties, as studied in \cite{HPS13} in the context of implicit-explicit (IMEX) Runge-Kutta methods for stiff optimal control problems. Note however that, already for initial value ODEs, such implicit methods can become very costly for nonlinear stiff problems in large dimension, requiring the usage of Newton-type methods and sophisticated linear algebra tools (preconditioners, etc.).
Alternatively to using implicit methods, in this paper we focus on fully explicit indirect methods, and introduce new families of explicit stabilized methods for stiff optimal control problems. The proposed methods rely on the so-called Runge-Kutta-Chebyshev methods of order one and its extension RKC of order two \cite{HoS80}.
Such explicit stabilized methods are popular in the context of initial value problems of stiff differential equations, particularly in high dimensions in the context of diffusive PDEs, see e.g. the survey \cite{Abd13c}.
It was extended to the stochastic context first in \cite{AbC08,AbL08} and recently in \cite{AAV18} for the design of explicit stabilized integrators with optimally large stability domains in the context of mean-square stable stiff and ergodic problems.

This paper is organized as follows. In Section \ref{sec:pre}, we recall standard tools on explicit stabilized methods and classical results on standard Runge-Kutta methods applied to optimal control problems.
In Section \ref{sec:dasection}, we introduce the new explicit stabilized schemes for optimal control problems and analyze their convergence and stability properties. Finally, Section \ref{sec:num} is dedicated to the numerical experiments where we illustrate the efficiency of the new approach.

\section{Preliminaries}\label{sec:pre}
\subsection{Discretization, order conditions, and symplecticity}
%
Let us first recall the definition of a Runge-Kutta method for ordinary differential equations (ODEs),
\begin{equation} \label{eq:ode}
	\dot y(t) = F(y(t)), \qquad y(0)=y^0,
\end{equation}
where $y:[0,T]\to\IR^n$ is the unknown solution, $F:\IR^n\to\IR^n$ is a smooth vector field, and $y^0\in\IR^n$ is a given initial condition. 
We consider for simplicity a uniform discretization of the interval $[0,T]$ with $N+1$ points for $N\in\N$, and denote by $h=T/N$ the stepsize. For a given integer $s$ and given real coefficients $b_i,~a_{ij}~(i,j=1,\dots,s)$, an $s$-stage Runge-Kutta method, $y_k\approx y(t_k),~t_k=kh$, to approximate the solution of \eqref{eq:ode}, is defined, for all $k=0,\dots,N-1$, by
\begin{equation}
	\label{eq:rk}
	\begin{split}
		y_{ki}&=y_k+h\sum_{j=1}^{s}a_{ij}F(y_{kj}),\quad i=1,\dots,s,\qquad y_{k+1}=y_k+h\sum_{i=1}^{s}b_iF(y_{ki}).
	\end{split}
\end{equation} 
The coefficients are usually displayed in a Butcher tableau as follows
\begin{equation}\label{eq:butcher}
	\begin{array}
		{c|c}
		& a_{ij}\\
		\hline
		& b_i
	\end{array},
\end{equation}
and we will sometimes use the notation $(a_{ij},b_i)$.
For more details about the order conditions of Runge-Kutta methods in the context of initial value ODEs, we refer for example to the book \cite[Chap.\ts III]{HLW06}.
We denote by $y_{k+1}=\Phi_h(y_k)$ the numerical flow of \eqref{eq:rk}, while the time adjoint method $\Phi^*_h$ of $\Phi_h$ is the inverse map of the original method with reversed time step $-h$, i.e., $\Phi^*_h:=\Phi^{-1}_{-h}$ \cite[Sect.\ts II.3]{HLW06}.
We recall that the time adjoint of an $s$-stage Runge-Kutta method $(a_{ij},b_i)$ \eqref{eq:rk} is again an $s$-stage Runge-Kutta method with the same order of accuracy and its coefficients $(a^*_{ij},b^*_i)$ are given by
$
a^*_{ij}=b_{s+1-j}-a_{s+1-i,s+1-j} \text{ and } b^*_i=b_{s+1-i}, \text{ where } i,j=1,\dots,s.
$

If we discretize \eqref{eq:P} using a Runge-Kutta discretization as above we naturally get the following discrete optimization problem,
\begin{equation}\label{eq:dp} 
	\begin{split}
		\min ~~ &\Psi(y_N); \quad \text{ subject to: }\\
		y_{k+1}&=y_k+h\sum_{i=1}^{s}b_if(u_{ki},y_{ki}),\qquad y_{ki}=y_k+h\sum_{j=1}^{s}a_{ij}f(u_{kj},y_{kj}),
	\end{split}
\end{equation}
where $i=1\dots,s,~ k=0,\dots,N-1,~\text{ and } y_0=y^0$. We denote by $p_{ode}$ the order of accuracy of the method \eqref{eq:rk} applied to the ODE problem \eqref{eq:ode} and by $p_{oc}$ the order of the method \eqref{eq:dp} for solving the optimal control problem \eqref{eq:P}. Note that we always have $p_{oc} \leq p_{ode}$. In general, $p_{oc} < p_{ode}$ because additional order conditions, described in \cite{Hager00,BL06}, have to be satisfied.

Let us denote by $H(u,y,p):=p^Tf(u,y)$ the pseudo-Hamiltonian of the system where $p$ is the Lagrange multiplier (or the costate) associated to the state $y$. Applying Pontryagin's maximum (or minimum) principle, the first order optimality conditions of \eqref{eq:P} are given by the following boundary value problem,
\begin{equation}\label{eq:oc} 
	\begin{split}
		\dot y(t)&=f(u(t),y(t))=\nabla_pH(u(t),y(t),p(t)),\\
		\dot p(t)&=-\nabla yf(u(t),y(t))p= -\nabla_yH(u(t),y(t),p(t)),\\
		0&=\nabla_uH(u(t),y(t),p(t)).\\
		t&\in [0,T],\quad y(0)=y^0,\quad p(T)=\nabla \Psi(y(T)).
	\end{split}
\end{equation}
Applying a Runge-Kutta integrator { naively to \eqref{eq:oc}} as an initial value system of ODEs combined with the classical methodology of shooting methods, would lead to severe instability due to the forward in time integration of the costate equation. For instance, for the optimal control of a diffusion PDE problem such as  $\partial_t y(t,x) = \Delta y(t,x) +u(t,x)$, where $\Delta$ is the Laplace operator, $y(t,x)$ is the state function, and $u(t,x)$ is the control function  (see also the diffusion-convection PDE problem considered in Sect. \ref{sec:burger}), then the costate equation takes the form of a heat equation with the wrong sign, $\partial_t p(t,x) = - \Delta p(t,x)$, which is naturally unstable if integrated forward in time. This makes classical { shooting methods} not applicable in the context of stiff dissipative optimal control problems considered in this paper. Alternatively, we consider a forward-backward iterative algorithm as described below (Algorithm \ref{algo:iter}).

Introducing Lagrange multipliers for the finite dimensional optimization problem \eqref{eq:dp}, and supposing that $b_i\neq0$ for all $i=1\dots,s$, a calculation \cite{Hager00,BL06} yields the following discrete optimality conditions
\begin{equation}\label{eq:doc}
	\begin{split}
		y_{k+1} &= y_k+h\sum_{i=1}^{s}b_if(u_{ki},y_{ki}),\quad\ 
		y_{ki} = y_k+h\sum_{j=1}^{s}a_{ij}f(u_{kj},y_{kj}),\\
		p_{k+1} &= p_k-h\sum_{i=1}^{s}\hat{b}_i\nabla_yH(u_{ki},y_{ki},p_{ki}),~ 
		p_{ki} = p_k-h\sum_{j=1}^{s}\hat{a}_{ij}\nabla_yH(u_{kj},y_{kj},p_{kj}),\\
		~~0&=\nabla_uH(u_{ki},y_{ki},p_{ki}),\quad k=0,\dots,N-1\quad i=1,\dots,s,\\
		~~y_0&=y^0,\quad p_N=\nabla \Psi(y_N)
	\end{split}
\end{equation}	
where the coefficients $\hat{b}_i$ and $\hat{a}_{ij}$ are defined by the following relations which, as observed in \cite{BL06}, correspond to the symplecticity conditions of partitioned Runge-Kutta methods for ODEs,
\begin{equation}\label{eq:symp}
	\hat{b}_i:=b_i,\qquad \hat{a}_{ij}:=b_j-\frac{b_j}{b_i}a_{ji},\qquad i=1,\dots s, j=1,\dots,s.
\end{equation}
Note that the vectors $p_k$ and $p_{ki}$ are the Lagrange multipliers associated to $y_k$ and $y_{ki}$ respectively. 
Assuming that the Hessian matrix $\nabla_u^2H(u,y,p) \in \R^{m\times m}$ is invertible along the trajectory of the exact solution, by the implicit function theorem there exists a $C^\infty$ function $\phi$ such that $u=\phi(y,p)$ and then \eqref{eq:doc} is equivalent to a partitioned Runge-Kutta (PRK) method. As noticed in \cite{BL06}, if we consider the problem \eqref{eq:oc} as a Hamiltonian system, with the Hamiltonian function $\mathcal{H}(y,p):=H(\Psi(y,p),y,p)$, then the obtained PRK \eqref{eq:doc} scheme is symplectic thanks to the relations \eqref{eq:symp}.

\begin{theorem}[Theorem 4.1 in \cite{Hager00}]
	\label{th:hager}
	Consider a Runge-Kutta method $(a_{ij},b_i)$ of order $p_{ode}$ for ODEs, where $b_i\neq0$, for all $i=1,\ldots,s$, applied to the optimal control problem \eqref{eq:P}. Consider the optimality conditions \eqref{eq:oc}, and assume that $\nabla_u^2H(u,y,p)$ is invertible in a neighborhood of the solution, then we have the following theorem.
	If we discretize \eqref{eq:oc} using an $s$-stage partitioned Runge-Kutta method $(a_{ij},b_i)-(\hat a_{ij},\hat b_i)$ of order $p^{\star}_{ode}$ for ODEs (as partitioned RK method), and the condition \eqref{eq:symp} is satisfied, then the order $p_{oc}$ of \eqref{eq:dp} satisfies $p_{oc}=p^{\star}_{ode}\leq p_{ode}$ and the schemes \eqref{eq:dp} and \eqref{eq:doc} are equivalent. In particular, for $p_{ode}\geq2$, equivalently $\sum_{i=1}^{s}b_i=1$ and $\sum_{i,j=1}^{s}b_ia_{ij}=\frac12$, we get $p^{\star}_{ode}\geq2$ and $p_{oc}\geq2$.
	\label{th:diagram}
\end{theorem}

The proof of Theorem \ref{th:diagram} relies on the  commutativity of the following diagram \cite[Sect.\ts2]{BL06} which means that  methods \eqref{eq:dp} and \eqref{eq:doc}, {color{red}where \eqref{eq:doc} is a symplectic partitioned Runge-Kutta method}, yield exactly the same outputs (up to round-off errors) if derived for Runge-Kutta discretizations of \eqref{eq:P} and \eqref{eq:oc} respectively. {  We also refer to the article \cite{SS16} where the role of symplectic partitioned Runge-Kutta methods involved in this commutative diagram is discussed.}
\begin{equation*}
	\begin{tikzpicture}
		\matrix (m) [matrix of math nodes,row sep=2em,column sep=5em,minimum width=3em]
		{
			\eqref{eq:P} & \eqref{eq:dp} \\
			\eqref{eq:oc} & \eqref{eq:doc} \\};
		\path[-stealth]
		(m-1-1) edge node [left] {optimality conditions} (m-2-1)
		edge node [above] {discretization} (m-1-2)
		(m-2-1.east|-m-2-2) edge node [below] {discretization}
		node [above] {} (m-2-2)
		(m-1-2) edge node [right] {optimality conditions} (m-2-2);
	\end{tikzpicture}
\end{equation*}

Remark that in \eqref{eq:doc}, if the method $(a_{ij},b_i)$ is explicit, then $(\hat a_{ij},\hat b_i)$ is in contrast an implicit method. Hence it is useful to consider the costate equation backward in time and use the time adjoint of $(\hat a_{ij},\hat b_i)$ which turns out to be explicit as shown in Proposition \ref{prop:exp} in Section \ref{sec:dasection}.  
Indeed, 
consider method \eqref{eq:doc} and proceed as in \cite{Hager00,HPS13},
$$
p_{k+1}+h\sum_{i=1}^{s}\hat{b}_i\nabla_yH(u_{ki},y_{ki},p_{ki})=p_{ki}+h\sum_{j=1}^{s}\hat{a}_{ij}\nabla_yH(u_{kj},y_{kj},p_{kj}),
$$
we then deduce from the identity $b_j-\hat{a}_{ij} = \frac{b_j}{b_i}a_{ji}$
the following formulation where $p_N$ serves to initialize the induction on $k=N-1,  \ldots, 0$,
$$
p_k=p_{k+1}+h\sum_{i=1}^{s}b_i\nabla_yH(u_{ki},y_{ki},p_{ki}),\quad
p_{ki} =p_{k+1}+h\sum_{j=1}^{s}\frac{b_j}{b_i}a_{ji}\nabla_yH(u_{kj},y_{kj},p_{kj}).
$$
The above Runge-Kutta method $(\tilde a_{ij},\tilde b_i):=(\frac{b_j}{b_i}a_{ji},b_i)$ for the costate is in fact the time adjoint of $(\hat a_{ij},\hat b_i)$. Since the method $(\hat a_{ij},\hat b_i)$ is called in the literature the adjoint method in the sense of optimal control because it is applied to the adjoint equation (costate), we call the method $(\tilde a_{ij},\tilde b_i):=(\frac{b_j}{b_i}a_{ji},b_i)$ the \textbf{double adjoint} of $(a_{ij},b_i)$, and we rewrite method \eqref{eq:doc} as
\begin{equation}\label{eq:doc1}
	\begin{split}
		y_{k+1} &= y_k+h\sum_{i=1}^{s}b_if(u_{ki},y_{ki}),\quad k=0,\dots,N-1\\
		y_{ki} &= y_k+h\sum_{j=1}^{s}a_{ij}f(u_{kj},y_{kj}),\quad k=0,\dots,N-1\quad i=1,\dots,s\\
		p_{k} &= p_{k+1}+h\sum_{i=1}^{s}{\tilde b_i}\nabla_yH(u_{ki},y_{ki},p_{ki}),\quad k=N-1,\dots,0\\
		p_{ki} &= p_{k+1}+h\sum_{j=1}^{s}{\tilde a_{ij}}\nabla_yH(u_{kj},y_{kj},p_{kj}),\quad k=N-1,\dots,0\quad i=s,\dots,1\\
		~~0&=\nabla_uH(u_{ki},_{ki},p_{ki}),\qquad k=0,\dots,N-1\quad i=1,\dots,s\\
		~~y_0&=y^0,\qquad p_N=\nabla \Psi(y_N).
	\end{split}
\end{equation}
Note that we integrate the state forward in time (increasing indices $k$) and the costate backward in time (decreasing $k$).

{An immediate consequence of Theorem \ref{th:hager} is that applying naively a Runge-Kutta method yields in general an order reduction, as stated in the following remark.
	\begin{remark} 
		Consider a Runge-Kutta method $(a_{ij},b_i)$ of order $p_{ode}=2$ and define $(\tilde a_{ij},\tilde b_i):=(a_{ij},b_i)$, in general the obtained partitioned Runge-Kutta method \eqref{eq:doc1} is not of order $p_{oc}=2$. Indeed, the coupling order conditions $\sum_{i,j=1}^{s}b_i\hat a_{ij}=\frac12$ and $\sum_{i,j=1}^{s}\hat b_i a_{ij}=\frac12$ are not automatically satisfied in general. In particular, for $(a_{ij},b_i)$ being the standard order two RKC method studied in the next section below, it can be checked that $p_{oc}=1$. This makes non trivial the construction of an explicit stabilized scheme of order 2 for optimal control problems, as described in section \ref{sec:rkcda}. We will see that the notion of double adjoint of a Runge-Kutta method, as described above, is an essential tool in our study.
		
\end{remark}}

To implement \eqref{eq:doc1} (equivalent to \eqref{eq:doc}), we shall use the following classical iterative algorithm which was proposed as a parallel algorithm with $N$ sub-problems in \cite[Algo.\ts 4]{MRS13}. For simplicity of the presentation, we only recall the non parallel algorithm, but emphasize that the parallel version could also be used in our context with explicit stabilized schemes. 
\begin{algo} (see for instance \cite[Algo.\ts 4]{MRS13}). \label{algo:iter}
	First start with an initial guess for the internal stages of the control $U^0=(u^0_{ki})_{k=0\dots,N-1}^{i=1,\dots,s}$ where $u^0_{ki}\in\IR^{m}$ for all $k$ and $i$. Denote by $y^l=(y^l_k)_{k=0,\dots,N-1}$ and $Y^l=(y^l_{ki})_{k=0\dots,N-1}^{i=1,\dots,s}$ the collection of the state values and its internal stages respectively at iteration $l$, and analogously we use the notations $p^l$ and $P^l$ for costate and its internal stages at iteration $l$. Suppose that at the iteration $l$, $U^l$ is known. For the next iteration $l+1$, the computation of $U^{l+1}$ is achieved as follows.
	\begin{enumerate}
		\item Compute $Y^l,~P^l,y^l,~p^l$ as in \eqref{eq:doc1}, the computation is done forward in time for the state $y_k$ and backward in time for the costate $p_k$.
		\item Compute $\tilde u^{l+1}_{ki}$ solving the system 
		$\nabla_uH(\tilde u^{l+1}_{ki},
		y^l_{ki},p^l_{ki})=0$, for all $k$ and $i$ using an analytical formula if available, or a Newton method for instance.
		\item Denoting $\tilde U^l=(\tilde u^l_{ki})_{k=0\dots,N-1}^{i=1,\dots,s}$, define $U^{l+1}$ by $U^{l+1}=(1-\theta^l)U^l+\theta^l\tilde U^{l+1}$, where $\theta^l$ is defined to minimize the scalar function $\theta\mapsto \Psi(U^{l+1})$ using a simple trisection method for instance, where $\Psi:U\mapsto\Psi(y^N)$.
	\end{enumerate}
	We stop when $\|U^{l+1}-U^l\|\leq tol$, where $tol$ is a prescribed small tolerance.
\end{algo}

In the recent paper \cite{LF19}, it was shown that the forward-backward sweep iteration defined in Algorithm \ref{algo:iter} used in implementing discretized optimal control problems converges when using a symplectic Runge-Kutta discretization, which strengthens the interest of such symplectic methods.

For simplicity, we assume in the rest of the paper that $y_0=y^0$ always holds in \eqref{eq:doc}.

\subsection{Explicit stabilized methods}\label{sec:expstab}

Stability is a crucial property of numerical integrators for solving stiff problems and we refer to the book \cite{hairer96sod}.
A Runge-Kutta method is said to be stable if the numerical solution stays bounded along the integration process.	Applying a Runge-Kutta method \eqref{eq:rk} to the linear test problem (with fixed parameter $\lambda\in \mathbb{C}$),
\begin{equation}\label{eq:test}\dot y=\lambda y,\quad y(0)=y_0,\end{equation}
with stepsize $h$ yields a recurrence of the form $y_{k+1}=R(h\lambda)y_k$ and by induction we get $y_k=R(h\lambda)^ky_0$. The function $R(z)$ is called the stability function of the method and the stability domain is defined as 
$
{\cal S}:=\{z\in \mathbb{C}; |R(z)|\leq 1\},
$
and ${y_k}$ remains bounded if and only if $h\lambda\in\cal S$.
The same result also applies to the internal stages of the Runge-Kutta method, for all $i=1,\dots,s$, where $s$ is the number of internal stages, $y_{ki}=R_i(h\lambda)y_k$, for some function $R_i$.
Remark that $R(z)$ is a rational function for implicit methods, but in the case of explicit methods the stability function $R(z)$ reduces to a polynomial.
The simplest Runge-Kutta type method to integrate ODEs \eqref{eq:ode}
is the explicit Euler method $y_{k+1}=y_{k}+hf(y_k)$ with stability polynomial $R(z)=1+z$. However, its stability domain ${\cal S}$ is small (it reduces to the disc of center $-1$ and radius $1$ in the complex plane) which yields a severe time step restriction and makes it very expensive for stiff problems. 

\subsubsection{Optimal first order Chebyshev methods}
The idea of explicit stabilized methods (as introduced in \cite{HoS80}, see the survey \cite{Abd13c}) is to construct explicit Runge-Kutta integrators with extended stability domain that grows quadratically with the number of stages $s$ of the method along the negative real axis, and hence allows to use large time steps typically for problems arising from diffusion partial differential equations. The family of methods considered in \cite{HoS80} is known as ``Chebyshev methods'' since its construction relies on Chebyshev polynomials $T_s(x)$ satisfying $T_s(\cos \theta) =\cos (s \theta)$. These polynomials allow us to obtain a two-step recurrence formula and hence low memory requirements and good internal stability with respect to round-off errors. 
The order one Chebyshev method for solving a stiff ODE \eqref{eq:ode} is defined as an explicit $s$-stage Runge-Kutta method by the recurrence
\begin{eqnarray} 
	\label{eq:Cheb1} 
	y_{k0} &=& y_{k}, \quad
	y_{k1} = y_{k}+ \mu_1 h F(y_{k0}), \nonumber \\
	y_{ki} &=&\mu_ihF(y_{k,{i-1}})+\nu_iy_{k,{i-1}}+(1-\nu_i)y_{k,{i-2}},\quad j=2,\ldots,s\\
	y_{k+1} &=& y_{ks} \nonumber,
\end{eqnarray}
where 
\begin{equation}\label{eq:omega}
	\omega_{0}:=1+\frac{\eta}{s^{2}},\quad \omega_{1}:=\frac{T_{s}(\omega_{0})}{T'_{s}(\omega_{0})},\quad
	\mu_1:=\frac{\omega_1}{\omega_0},
\end{equation} 
{where $\eta$ is called the damping parameter and is used to make the stability of the method robust with respect to small perturbations as described below.} Finally, for all $i=2,\ldots,s,$
\begin{equation}\label{eq:coeffscheb}
	\mu_i:=\frac{2\omega_1T_{i-1}(\omega_0)}{T_i(\omega_0)},\quad 
	\nu_i:=\frac{2\omega_0T_{i-1}(\omega_0)}{T_i(\omega_0)}.
\end{equation}
One can easily check that the (family) of methods \eqref{eq:Cheb1} has the same first order of accuracy as the explicit Euler method (recovered for $s=1$).
Note that instead of the standard Runge-Kutta method formulation \eqref{eq:rk} with coefficients $(a_{ij},b_i)$, the one step method $y_{k+1}=\Phi_h(y_k)$ in \eqref{eq:Cheb1} should be implemented using a recurrence relation (indexed by $j$) inspired from the relation \eqref{eq:recTU1} on Chebyshev polynomials
\begin{equation} \label{eq:recTU1}
	T_j(z) = 2z T_{j-1}(z) - T_{j-2}(z),\qquad
\end{equation}
where $T_0(z)= 1,T_1(z)= z$. 
This implementation \eqref{eq:Cheb1}  yields a good stability \cite{HoS80} of the scheme with respect to round-off errors.
The most interesting feature of this scheme is its stability behavior. Indeed, the method \eqref{eq:Cheb1} applied to \eqref{eq:test} yields, 
with $z=\lambda h$,
$y_{k+1}=R_s^\eta(z) y_{k}=\frac{T_{s}(\omega_0+\omega_1 z)}{T_{s}(\omega_0)} y_{k}$. A large real negative interval $(-C_\eta s^2 ,0)$ is included in the 
stability domain of the method ${\cal S}:=\{z\in \mathbb{C}; |R_s^\eta(z)|\leq 1\}$.
For the internal stages, we have analogously
$y_{ki}=R_{s,i}^\eta(z) y_k =\frac{T_i(\omega_{0}+\omega_1z)}{T_i(\omega_{0})}y_k$. The constant $C_\eta = 2-4/3\,\eta +\bigo(\eta^2)$ depends on the so-called damping parameter $\eta$ and for $\eta=0$,
it reaches the maximal value $C_0=2$ (also optimal with respect to all possible stability polynomials for explicit schemes of order 1). 
Hence, given the stepsize $h$, for dissipative vector fields with a Jacobian having large real negative eigenvalues (such as diffusion problems) with spectral radius $\lambda_{\max}$ at $y_n$, the parameter $s$ for the next step $y_{n+1}$ can be chosen adaptively as\footnote{The notation $[x]$ stands for the integer rounding of real numbers.}
\begin{equation} \label{eq:defsnum}
	s := \left[\sqrt{\frac{h\lambda_{\max}+1.5}{2-4/3\,\eta}}+0.5\right],
\end{equation}
see \cite{Abd02} in the context of stabilized schemes of order two with adaptive stepsizes.
The method \eqref{eq:Cheb1}
is much more efficient as its stability domain increases {\it quadratically} with the number $s$ of function evaluations while a composition of $s$  explicit Euler steps
(same cost) has a stability domain that only increases {\it linearly} with $s$.
\begin{figure}[t]
	\centering
	\begin{subfigure}[t]{0.49\linewidth}
		\centering
		\includegraphics[width=\linewidth]{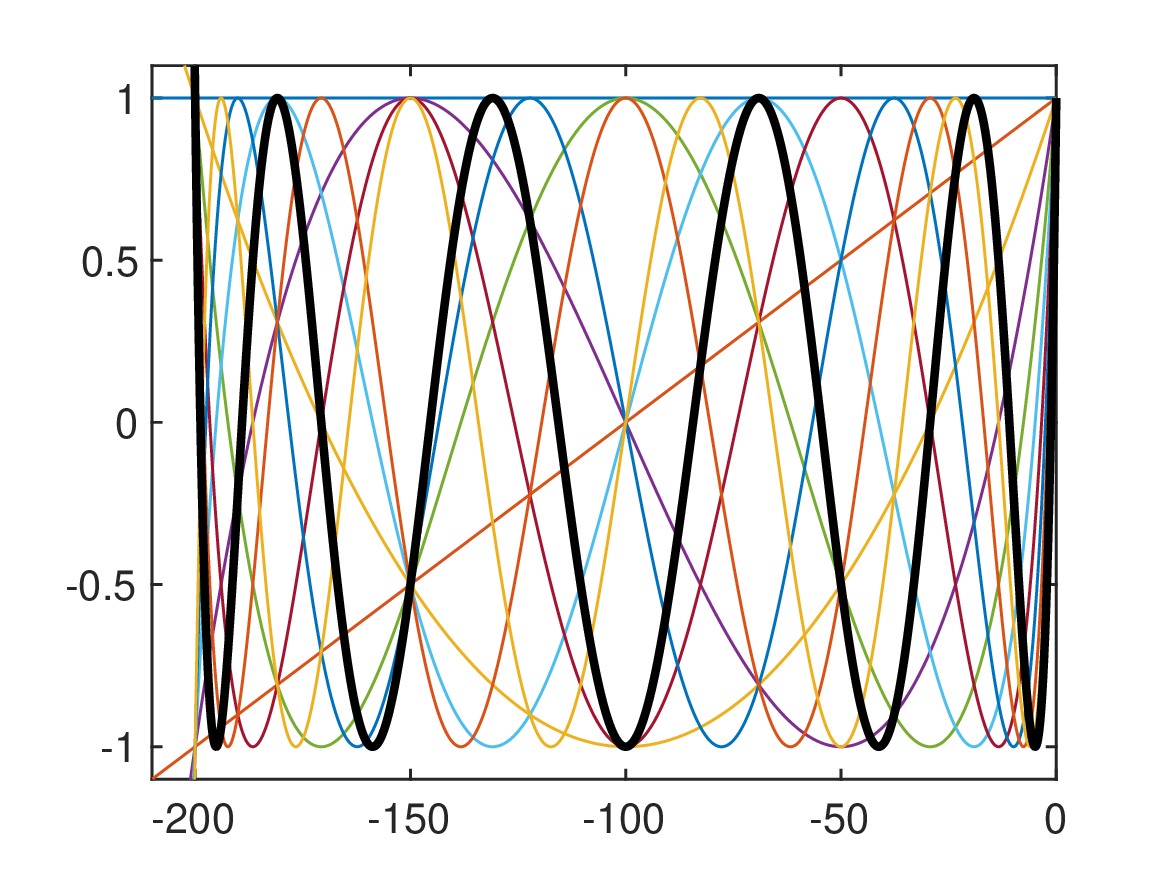}
		\caption{$\eta=0$}
	\end{subfigure}
	\begin{subfigure}[t]{0.49\linewidth}
		\centering
		\includegraphics[width=\linewidth]{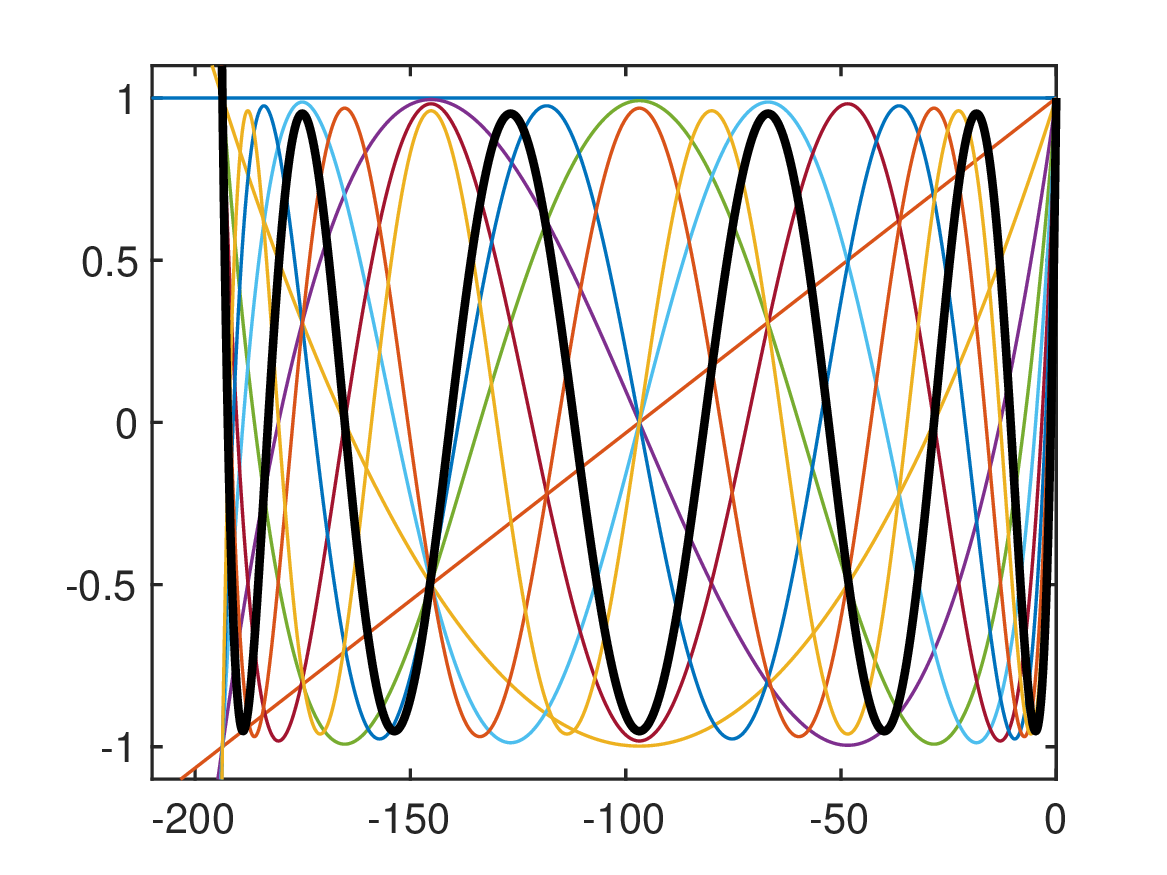}
		\caption{$\eta=0.05$}
	\end{subfigure}
	\caption{Internal stages (thin curves) and stability polynomials (bold curves) of the  Chebyshev method \eqref{eq:Cheb1} for $s=10$ with and without damping.}
	\label{fig:stagescheb}
\end{figure}
In Figure \ref{fig:stagescheb} we plot the internal stages 
for $s=10$ and different values $\eta=0$ and $\eta=0.05$ , respectively.
We observe that in the absence of damping ($\eta=0$), the stability function (here a polynomial) is bounded by $1$ in the large real interval $[-2s^2,0]$ of width $2\cdot 10^2=200$. However, for all $z$ that are local extrema of the stability function, where $|R_s^\eta(z)|=1$, the stability domain is very narrow in the complex plane. 
Here the importance of some damping appears, to make the scheme robust with respect to small perturbations of the eigenvalues. A typical recommended value for the damping parameter is $\eta=0.05$, see \cite{Ver96b,Abd13c}. The advantage of this damping is that the stability polynomial is now \textbf{strictly} bounded by $1$ and the stability domain includes a neighborhood of the negative interval $(-C_\eta s^2,0)$. This improvement costs a slight reduction of the stability domain length from $2s^2$ to $C_{\eta} s^2$ where $C_\eta\geq 2-\frac43 \eta$.

\subsubsection{Second order RKC methods}

To design a second order method, we need the stability polynomial to satisfy\footnote{Indeed, up to order two, the order conditions for nonlinear problems are the same as the order conditions for linear problems \cite[Chap.\ts III]{HLW06}.} $R(z)=1+z+\frac{z^2}2+\mathcal{O}(z^3)$. In \cite{B71}, Bakker introduced a correction to the first order shifted Chebyshev polynomials to get the following second order polynomial
\begin{equation} \label{eq:stabRKC}
	R_s^\eta (z) = a_s+b_sT_s(\omega_0+\omega_2z),
\end{equation}
where,
\begin{equation}\label{eq:omegarkc}
	a_s=1-b_sT_s(\omega_0),\quad
	b_s=\frac{T''_s(\omega_0)}{(T'_s(\omega_0)^2)},\quad
	\omega_0=1+\frac{\eta}{s^2},\quad
	\omega_2=\frac{T'_s(\omega0)}{T''_s(\omega_0)},\quad
	\eta=0.15.
\end{equation}
For each $s$, $|R_s^\eta(z)|$ remains bounded by $a_b+b_s=1-\eta/3 + \bigo(\eta^2)$ for $z$ in the stability interval (except for a small interval near the origin). The stability interval along the negative real axis is approximately $[-0.65s^2,0]$, and covers about $80\%$ of the optimal stability interval for second order stability polynomials, and the formula now for calculating $s$ for a given time step $h$ is
\begin{equation} \label{eq:defsnumrkc}
	s := \left[\sqrt{\frac{h\lambda_{\max}+1.5}{0.65}}+0.5\right].
\end{equation}
Using the recurrence relation of the Chebyshev polynomials, the RKC method as introduced in \cite{HoS80} is defined by
\begin{equation} 
	\label{eq:rkc}
	\begin{split}
		y_{k0} &= y_{k}, \quad
		y_{k1} = y_{k0}+h b_1\omega_2 f(y_{k0}), \\
		y_{ki} &=y_{k0}+\mu'_ih(f(y_{k,{i-1}})-a_{i-1}f({y_{k0}}))
		+\nu'_i(y_{k,{i-1}}-y_{k0})\\&+\kappa'_i(y_{k,{i-2}}-y_{k0}), \\
		y_{k+1} &= y_{ks},
	\end{split}
\end{equation}
where, 
\begin{equation}\label{eq:coeffsrkc}
	\mu'_i=\frac{2b_i\omega_2}{b_{i-1}},~~ 
	\nu'_i=\frac{2b_i\omega_0}{b_{i-1}},~~ 
	\kappa'_i=-\frac{b_i}{b_{i-2}},~~
	b_i=\frac{T''_i(\omega_0)}{(T'_i(\omega_0)^2)},~~
	a_i=1-b_iT_i(\omega_0),
\end{equation}
for $i=2,\ldots,s$. As in \eqref{eq:stabRKC}, the stability functions of the internal stages are given by $R_i^\eta (z) = a_i+b_iT_i(\omega_0+\omega_2z)$, where  $i=0,\dots,s-1,$ and the parameters $a_i$ and $b_i$ are chosen such that the above stages are consistent $R_i^\eta (z)=1+\mathcal{O}(z)$. The parameters $b_0$ and $b_1$ are free ($R_0^\eta(z)$ is constant and only order 1 is possible for $R_1^\eta$(z)) and the values $b_0=b_1=b_2$ are suggested in \cite{SV80}. 
In this paper, to facilitate the analysis of the internal stability of the optimal control methods, making the internal stages of the RKC method analogous to the Chebyshev method \eqref{eq:Cheb1} of order one, we introduce a new implementation of RKC method 
\begin{eqnarray} 
	\label{eq:rkcnew}
	y_{k0} &=& y_{k}, \quad
	y_{k1} = y_{k}+ \mu_1 h F(y_{k0}), \nonumber \\
	y_{ki} &=&\mu_ihF(y_{k,{i-1}})+\nu_iy_{k,{i-1}}+(1-\nu_i)y_{k,{i-2}},\quad i=2,\ldots,s\\
	y_{k+1} &=& a_sy_{k0}+b_sT_s(\omega_o)y_{ks} \nonumber,
\end{eqnarray}
where $\mu_1=\frac{\omega_2}{\omega_0}$, $a_s,b_s$ are given in \eqref{eq:omegarkc}, and the parameters $\mu_i$ and $\nu_i$ are defined by (analogously to \eqref{eq:coeffscheb}, using $\omega_2$ instead of $\omega_1$),
$
\mu_i=\frac{2\omega_2T_{i-1}(\omega_0)}{T_i(\omega_0)},~
\nu_i=\frac{2\omega_0T_{i-1}(\omega_0)}{T_i(\omega_0)}$, for $ i=2,\ldots,s.
$
This new formulation \eqref{eq:rkcnew} yields the same stability function $R^{\eta}_s(z)$ in \eqref{eq:stabRKC} but different internal stages, and it will be helpful when we introduce the double adjoint of RKC in Section \ref{sec:dasection}. 
We recall that for an accurate implementation, one should not use the standard Runge-Kutta formulations with coefficients $(a_{ij},b_i)$ for \eqref{eq:Cheb1} and \eqref{eq:rkc} since they are unstable due to the accumulated round-off error for large values of $s$. In contrast, the \emph{low memory} induction formulations \eqref{eq:Cheb1} and \eqref{eq:rkc} are easy to implement and very stable with respect to round-off errors~\cite{HoS80}. 
\begin{figure}
	\centering
	\begin{subfigure}[t]{0.49\linewidth}
		\centering
		\includegraphics[width=\linewidth]{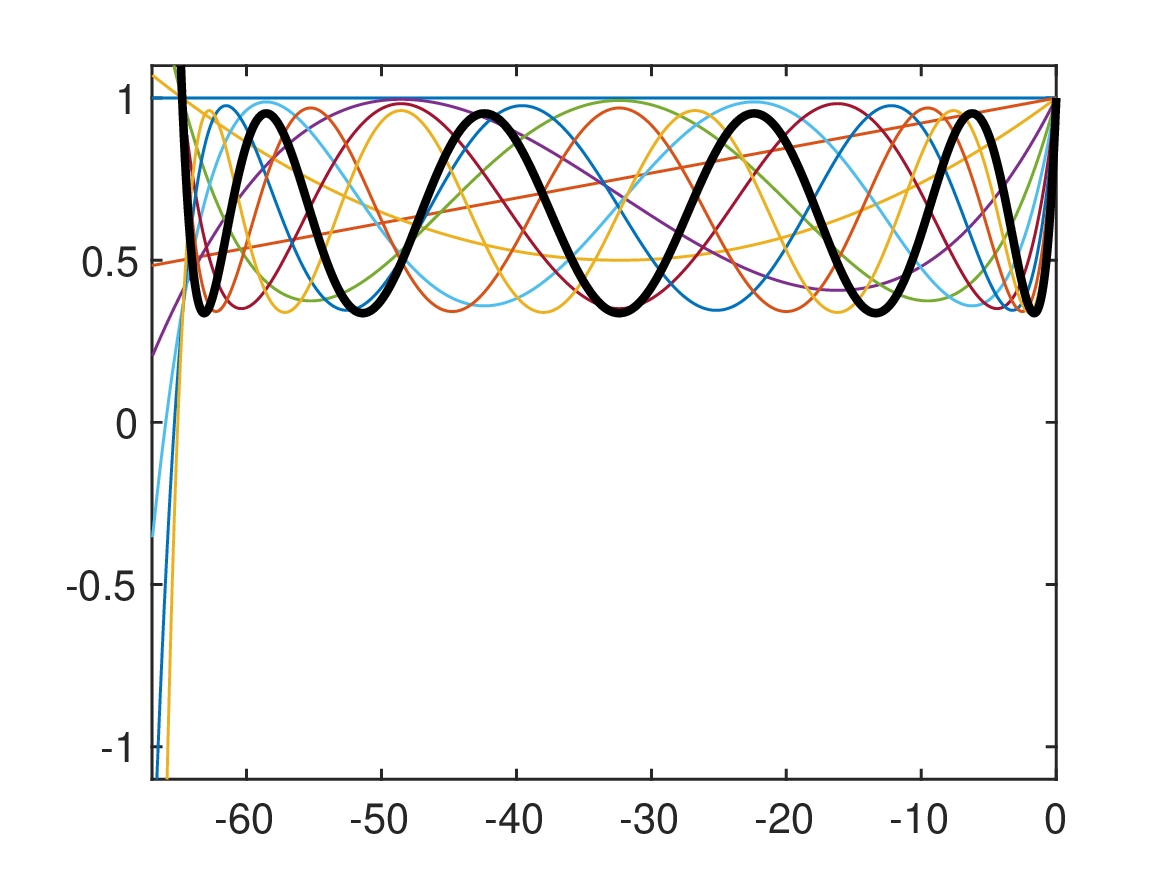}
		\caption{Classical RKC \eqref{eq:rkc}.}
	\end{subfigure}
	\begin{subfigure}[t]{0.49\linewidth}
		\centering
		\includegraphics[width=\linewidth]{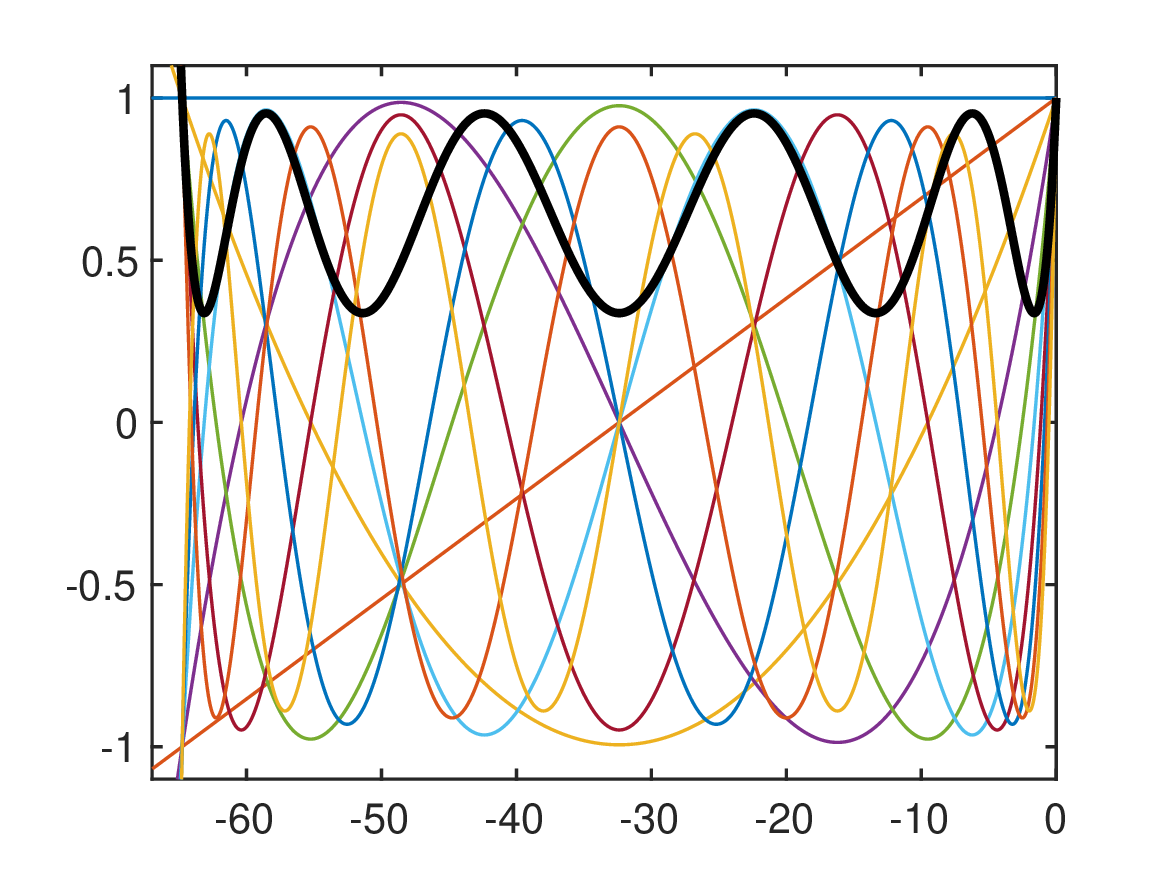}
		\caption{New RKC \eqref{eq:rkcnew}.}
	\end{subfigure}
	\caption{Internal stages (thin curves) and stability polynomials (bold curves) of the classical \eqref{eq:rkc} and the new \eqref{eq:rkcnew} RKC implementations for $s=10$ internal stages.}
	\label{fig:stagesrkc}
\end{figure}

Note that \eqref{eq:rkcnew} is not the same Runge-Kutta method as the standard RKC \eqref{eq:rkc} from \cite{HoS80}, it has different internal stages but the same stability function $R^{\eta}_s(z)$ in \eqref{eq:stabRKC} and hence order $p_{ode}=2$.
In Figure \ref{fig:stagesrkc} we can see that the internal stages of the new formulation \eqref{eq:rkcnew} of RKC have an analogous behavior compared to those of the first order Chebyshev method oscillating around zero (see Figure \ref{fig:stagescheb}), in contrast to those
of the standard RKC \eqref{eq:rkc}, oscillating around the value $a_s>0$. In addition, comparing Figures \ref{fig:stagesrkc}b and \ref{fig:stagescheb}b, we see that the internal stages of the new RKC method \eqref{eq:rkcnew} are the same as the order one Chebyshev method \eqref{eq:Cheb1} up to a horizontal rescaling $\omega_2/\omega_1$. This is because the $s$ internal stages of the methods have the stability function $T_i(\omega_0+\omega_jz)/T_i(\omega_0),~i=1,\dots,s$ for $j=1,2$ respectively.

Notice however that this modification of the standard RKC method \eqref{eq:rkc} deteriorates the order two of accuracy of the internal stages of the method, useful for PDEs with non homogeneous boundary conditions \cite[Chap.\ts V]{HV03}.
\begin{remark}
	Analogously to the standard RKC method \eqref{eq:rkc}, the new RKC formulation \eqref{eq:rkcnew} can be equipped with an error estimator to allow a variable time step control. Since the new formulation \eqref{eq:rkcnew} has the same stability function, one can use the same error estimator as proposed in \cite[Sect.\ts3.1]{SSV98}. In this paper we consider only a constant time step for simplicity of the presentation but emphasize that a variable time step $h_n$ can be used for the new optimal control method of order two.
\end{remark}
\begin{remark} \label{rem:ROCK2}
	Second order Runge-Kutta Orthogonal Chebyshev (ROCK2) methods, as introduced in \cite{AbM01}, are second order explicit stabilized methods for which the stability domain contains an interval that covers around $98\%$ of the optimal one for second order explicit methods. 
	It would be interesting to extend such second order methods with nearly optimally large stability domain to the context of optimal control problems.
	It turns out however that such an extension based on ROCK2 (or its order four extension ROCK4 \cite{Abd02}) is difficult and not analyzed in the present paper. This difficulty arises from the severe instability of the internal stages of the double adjoint of the standard ROCK2 method (see Figure \ref{fig:stagesrock2}) which would introduce large round-off errors for stiff problems (large values of $s$), making the obtained optimal control method not reliable.
\end{remark}
\begin{figure}
	\centering
	\begin{subfigure}[t]{0.49\linewidth}
		\centering
		\includegraphics[width=\linewidth]{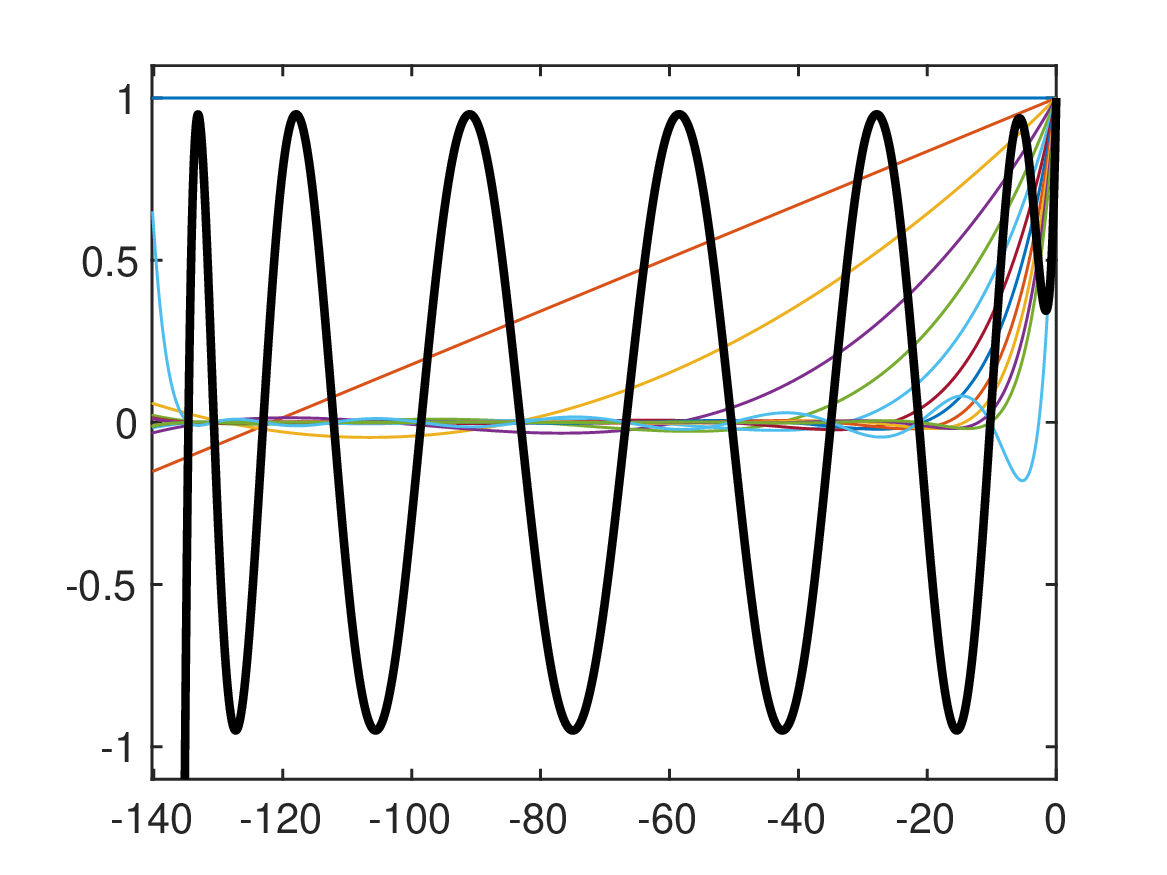}
		\caption{Classical ROCK2.}
	\end{subfigure}
	\begin{subfigure}[t]{0.49\linewidth}
		\centering
		\includegraphics[width=\linewidth]{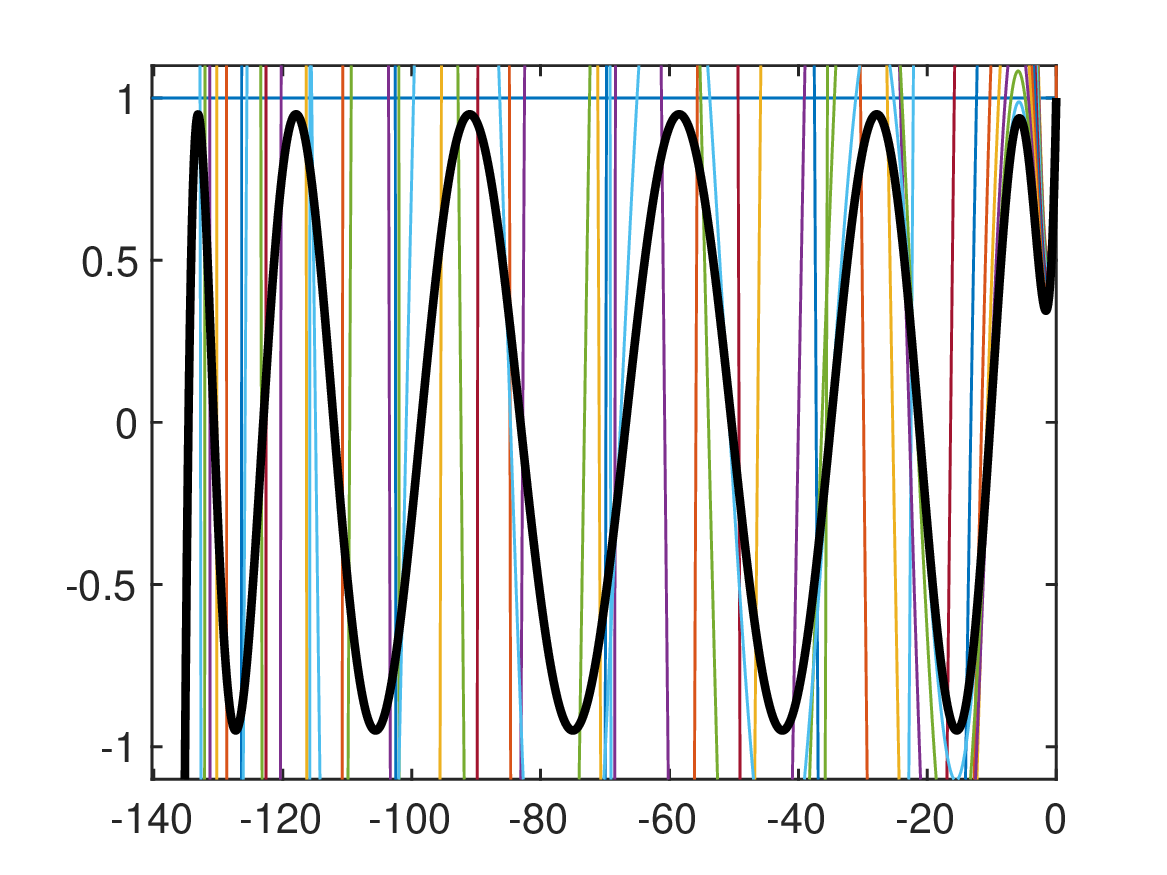}
		\caption{Double adjoint of ROCK2.}
	\end{subfigure}
	\caption{Internal stages (thin curves) and stability polynomials (bold curves) of the ROCK2 method \cite{AbM01} and its double adjoint for optimal control for $s=13$ stages.}
	\label{fig:stagesrock2}
\end{figure}

\section{Explicit stabilized methods for optimal control}\label{sec:dasection}

In this section, we derive new two term recurrence relations of the double adjoints of Chebyshev method \eqref{eq:Cheb1} and RKC method \eqref{eq:rkcnew} that are numerically stable. Indeed, one cannot rely on standard Runge-Kutta coefficients for the implementation of explicit stabilized schemes.

\subsection{Double adjoint of a general Runge-Kutta method}

Recall from \eqref{eq:doc1} that the Butcher tableau of the double adjoint $(\tilde a_{ij},\tilde b_i)$ of \eqref{eq:butcher} is
\begin{equation}\label{eq:da}
	\begin{array}
		{c|c}
		& \tilde a_{ji}\\
		\hline 
		&\\[-2ex]
		& \tilde b_i
	\end{array} :=
	\begin{array}
		{c|c}
		& \frac{b_j}{b_i}a_{ji}\\
		\hline
		&\\[-1.7ex]
		& b_i
	\end{array}.
\end{equation}
\begin{proposition}\label{prop:exp}
	If a Runge-Kutta method $(a_{ij},b_i)$ \eqref{eq:butcher} is explicit, then its double adjoint \eqref{eq:da} is explicit as well.
\end{proposition}
\begin{proof}
	For an explicit Runge-Kutta method we have that $a_{ij}=0$ for all $j\geq i$ i.e the matrix $(a_{ij})$ is strictly lower triangular. Permuting the internal stages in \eqref{eq:da} for $i,j=s,\ldots,2,1$ does not modify the method but yields the following Butcher tableau
	\begin{equation}
		\begin{array}
			{c|c}
			&\frac{b_{s+1-j}}{b_{s+1-i}}a_{s+1-j,s+1-i} \\
			\hline
			& b_{s+1-i}
		\end{array}
	\end{equation}
	which is strictly lower triangular, and thus the method \eqref{eq:da} is again explicit.
\end{proof}

An immediate consequence of Proposition \ref{prop:exp} is that for explicit methods, the stability function of the double adjoint is again a polynomial. In fact it turns out, as stated in Theorem \ref{th:stabfun} below, that for any Runge-Kutta method, the double adjoint $(\tilde a_{ij},\tilde b_i)$ has exactly the same stability function as $(a_{ij},b_i)$. Note however that this result does not hold in general for the internal stages (see Remark \ref{rem:ROCK2} about ROCK2).
\begin{theorem}\label{th:stabfun} A Runge-Kutta method $(a_{ij},b_i)$ and its double adjoint $(\tilde a_{ij},\tilde b_i)$ in \eqref{eq:da} share the same stability function $R(z)$. 
\end{theorem}
\begin{proof}
	Let $A=(a_{ij})$, $A_d=\left(a_{ji}\,b_j/b_i\right)$, and $b=(b_i)$, $i,j=1\dots s$. We recall the formula for the stability function of the Runge-Kutta method $(a_{ij},b_i)$,
	\begin{equation}\label{eq:stabproof}
		R(z)=\frac{\det(I-zA+z\mathds{1}b^T)}{\det(I-zA)},
	\end{equation}
	where $\mathds{1}\in\IR^s$ is the line vector of size $s$ containing only ones.
	Using a simple calculation, one can show that $A_d^T=DAD^{-1}$ where $D=\diag(b_i)$. This implies that $I-zA_d^T=D(I-zA)D^{-1}$, and since $I-zA_d^T=(I-zA_d)^T$, thus $\det(I-zA)=\det(I-zA_d)$.
	Using the same diagonal matrix $D$ we have $D\mathds{1}b^TD^{-1}=\mathds{1}^Tb$, hence $I-zA_d^T+z\mathds{1}b=D(I-zA+z\mathds{1}b)D^{-1}$ and $\det(I-zA_d^T+z\mathds{1}b)=\det(I-zA+z\mathds{1}b)$ and hence the stability function of $(\tilde a_{ij},\tilde b_i)$ is again \eqref{eq:stabproof}.
\end{proof}

\subsection{Chebyshev method of order one for optimal control problems}
For clarity of presentation, we first study Chebyshev method of order one for optimal control problems before introducing the second order RKC method.
Applying the order one Chebyshev method \eqref{eq:Cheb1} to the problem \eqref{eq:P} we get
\begin{equation} \label{eq:cheboc}
	\begin{split}
		\min&~\Psi(y_N),~such~that\\
		y_{k0} &= y_{k},\quad y_{k1} = y_{k0}+ \mu_1 h f(u_{k0},y_{k0}),\\
		y_{ki} &= \mu_ihf(u_{k,{i-1}},y_{k,{i-1}})+\nu_iy_{k,{i-1}}+(1-\nu_i)y_{k,{i-2}},\quad i=2,\ldots,s\\
		y_{k+1} &= y_{ks},
	\end{split}
\end{equation}
where, $k=0,\dots,N-1,~ \eta>0$ is fixed, and the parameters $\mu_i,~\nu_i$ are defined as in \eqref{eq:omega} and \eqref{eq:coeffscheb}.

For the implementation of Algorithm \ref{algo:iter} based on the order one Chebyshev method \eqref{eq:cheboc} for the state equation, the costate equation
can be implemented efficiently using the recurrence relations given by the following theorem.
\begin{theorem}
	The double adjoint of scheme \eqref{eq:cheboc} is given by the recurrence
	\begin{equation}\label{eq:chebda}
		\begin{split}
			p_N &= \nabla\Psi(y_N),\quad p_{ks}=p_{k+1}\\
			p_{k,{s-1}}&=p_{ks}+\frac{\mu_s}{\nu_s}h\nabla_yH(u_{k,{s-1}},y_{k,{s-1}},p_{ks})\\
			p_{k,{s-j}}&=\frac{\mu_{s-j+1}\alpha_{s-j+1}}{\alpha_{s-j}} h \nabla_yH(u_{k,{s-j}},y_{k,{s-j}},p_{k,{s-j+1}})\\&+\frac{\nu_{s-j+1}\alpha_{s-j+1}}{\alpha_{s-j}}p_{k,{s-j+1}}
			\\&+\frac{(1-\nu_{s-j+2})\alpha_{s-j+2}}{\alpha_{s-j}}p_{k,{s-j+2}},\quad j=2\dots,s-1,\\
			p_{k0}&= \mu_1\alpha_1h\nabla_yH(u_{k0},y_{k0},p_{k1})+\alpha_1p_{k1}+(1-
			\nu_2)\alpha_2p_{k2}\\
			p_k&=p_{k0}\\
			&\nabla_uH(u_{k,{s-j}},y_{k,{s-j}},p_{k,{s-j+1}})=0,\quad j=1,\dots,s.
		\end{split}
	\end{equation}
	where $k=N-1,\dots,2,1,0$ and the coefficients $\alpha_j$ are defined by induction as
	\begin{equation}
		\begin{split}
			\alpha_s&=1,\quad
			\alpha_{s-1}=\nu_s,\\
			\alpha_{s-j}&=\nu_{s-j+1}\alpha_{s-j+1}+(1-\nu_{s-j+2})\alpha_{s-j+2},~~j=2\dots s-1.
		\end{split}
	\end{equation}
	\label{th:cheb}
\end{theorem}
\vspace{-0.2cm}The proof of Theorem \ref{th:cheb} uses similar arguments to the proof of Theorem \ref{th:diagram}, with the exception that we now rely on the recurrence formula \eqref{eq:cheboc} instead of the standard Runge-Kutta formulation \eqref{eq:oc} to avoid numerical instability.
\begin{proof}[Proof of Theorem \ref{th:cheb}]
	The Lagrangian associated to the discrete optimization problem \eqref{eq:cheboc} is
	\begin{equation*}
		\begin{split}
			\mathcal{L} = \Psi(y_N)+p_0 \cdot (y^0-y_0)
			&+ \sum_{k=0}^{N-1} 
			\Big\{p_{k+1} \cdot (y_{ks}-y_{k+1})-p_{k0}\cdot (y_k-y_{k0})\\
			&+p_{k1}\cdot (y_{k0}+\mu_1 hf(u_{k0},y_{k0})-y_{k1})\\
			&+\sum_{i=2}^{s}p_{ki}\cdot (\mu_i h f(u_{k,{i-1}},y_{k,{i-1}})+\nu_iy_{k,{i-1}}\\&+(1-\nu_i)y_{ki-2}-y_{ki})\Big\}.
		\end{split}
	\end{equation*}
	Here $p_{k+1},~p_{ki},$ and $p_0$ are the Lagrange multipliers. The optimality necessary conditions are thus given by
	\begin{equation}\label{eq:lag}
		\dif{\L}{y_k}=0,~~~\dif{\L}{y_{ki}}=0,~~~\dif{\L}{p_k}=0,~~~\dif{\L}{p_{ki}}=0,~~~\dif{\L}{u_{ki}}=0,
	\end{equation}
	where $k=0,\dots,N-1$ and $i=0,\dots,s$. By a direct calculation, we obtain the following system,
	\begin{equation}\label{eq:chebda0}
		\begin{split}
			y_{k0}&=y_k,\quad
			y_{k1}=y_{k0}+\mu_1 hf(u_{k0},y_{k0}),\\
			y_{ki}&=\mu_ihf(u_{k,{i-1}},y_{k,{i-1}})+\nu_iy_{k,{i-1}}+(1-\nu_i) y_{k,{i-2}},\quad i=2,\ldots,s,\\
			y_{k+1}&=y_{ks},\\
			p_N &= \nabla\Psi(y_N),\quad p_{ks}=p_{k+1},\\
			p_{k,{s-1}}&=\mu_sh\nabla_yH(u_{k,{s-1}},y_{k,{s-1}},p_{ks})+\nu_sp_{ks},\\
			p_{k,{s-j}}&=\mu_{s-j+1} h \nabla_yH(u_{k,{s-j}},y_{k,{s-j}},p_{k,{s-j+1}})+\nu_{s-j+1}p_{k,{s-j+1}}\\&+(1-\nu_{s-j+2})p_{k,{s-j+2}},\quad j=2,\dots,s-1,\\
			p_{k0}&= \mu_1h\nabla_yH(u_{k0},y_{k0},p_{k1})+p_{k1}+(1-
			\nu_2)p_{k2},\\
			p_k&=p_{k0},\\
			&\nabla_uH(u_{k,{s-j}},y_{k,{s-j}},p_{k,{s-j+1}})=0,\quad j=1,\dots,s,
		\end{split}
	\end{equation}
	where $k=0,\dots,N-1$. In the above system, observe that the steps $p_{ki}$ of the double adjoint are not internal stages of a Runge-Kutta method, that is because they are not $\mathcal{O}(h)$ perturbations of the $p_{k+1}$, i.e. $p_{ki} \neq p_{k+1}+\bigo(h)$, for instance, already for the first step $p_{k,{s-1}}=\nu_sp_{k+1}+\bigo(h)$ with $\nu_s=2+\bigo(\eta)$. Since the pseudo-Hamiltonian $H(u,y,p)$ is linear in $p$, we can rescale the internal stages of the costate by a factor $\alpha_j$ such that for $\hat p_{kj}:=\alpha_j^{-1}p_{kj}$, we obtain $\hat p_{kj}=p_{k+1}+\bigo(h)$.
	We define
	\begin{equation*}
		\hat{p}_{ks}:=p_{ks},\quad
		\hat{p}_{k,{s-1}}:=\frac{p_{k,{s-1}}}{\nu_s}=\hat p_{ks}+\frac{\mu_s}{\nu_s}h\nabla_yH(u_{k,{s-1}},y_{k,{s-1}},\hat p_{ks}).
	\end{equation*}
	Substituting $\hat p_{k,{s-2}}$ in \eqref{eq:chebda0}, we obtain
	$$p_{k,{s-2}}=\mu_{s-1}\nu_sh\nabla_yH(u_{k,{s-2}},y_{k,{s-2}},\hat p_{k,{s-1}})+\nu_{s-1}\nu_s \hat p_{k,{s-1}}+(1-\nu_s)\hat p_{ks},$$
	the quantities $\hat p_{k,{s-1}}$ and $\hat p_{ks}$ are equal to $p_{k+1}+\bigo(h)$, hence $p_{k,{s-2}}=(\nu_{s-1}\nu_s+1-\nu_s)p_{k+1}+\bigo(h)$, this implies that $\alpha_{s-2}=\nu_{s-1}\nu_s+1-\nu_s$ and therefore
	$$\hat p_{k,{s-2}} := \frac{p_{k,{s-2}}}{\nu_{s-1}\nu_s+1-\nu_s}=\frac{p_{k,{s-2}}}{\nu_{s}(\nu_{s-1}-1)+1}.$$
	Following this procedure for $p_{k,{s-j}}, j=2,\dots,s-1$, we arrive at the Runge-Kutta formulation \eqref{eq:chebda} of the double adjoint of scheme \eqref{eq:cheboc}, where we go back to the notation $p_{ki}$ instead of $\hat p_{ki}$. 
\end{proof}
\begin{remark} \label{rem:nodamp}
	A straightforward calculation yields that without damping (for $\eta=0$), we have 
	$\alpha_{s-j}=j+1$, and
	method \eqref{eq:cheboc}-\eqref{eq:chebda} reduces to the following recurrence 
	\begin{equation}\label{eq:eta0}
		\begin{split}
			y_{k0}&=y_k,\quad
			y_{k1}=y_{k0}+\frac{h}{s^2}f(u_{k0},y_{k0}),\\
			y_{ki}&=\frac{2h}{s^2}f(u_{k,{i-1}},y_{k,{i-1}})+2y_{k,{i-1}}-y_{k,{i-2}},\quad i=2,\ldots,s,\\
			y_{k+1}&=y_{ks},\\
			p_N &= \nabla\Psi(y_N),\quad p_{ks}=p_{k+1},\\
			p_{k,{s-1}}&=p_{ks}+\frac{h}{s^2}\nabla_yH(u_{k,{s-1}},y_{k,{s-1}},p_{ks}),\\
			p_{k,{s-j}}&=\frac{2j}{(j+1)s^2} h \nabla_yH(u_{k,{s-j}},y_{k,{s-j}},p_{k,{s-j+1}})+\frac{2j}{j+1}p_{k,{s-i+1}}\\&+\frac{1-j}{j+1}p_{k,{s-j+2}},\quad j=2,\dots,s-1, \\
			p_{k0}&= \frac{h}{s}\nabla_yH(u_{k0},y_{k0},p_{k1})+sp_{k1}+(1-s)p_{k2},\\
			p_k&=p_{k0},\\
			&\nabla_uH(u_{k,{s-j}},y_{k,{s-j}},p_{k,{s-j+1}})=0,\quad j=1,\dots,s.
		\end{split}
	\end{equation}
	where $k=0,\dots,N-1$. In Section \ref{sec:stability}, we shall study the stability of \eqref{eq:eta0} (without damping) and of \eqref{eq:cheboc}-\eqref{eq:chebda} (with damping). 
\end{remark}

\subsection{RKC method of order 2}\label{sec:rkcda}

We consider the new implementation \eqref{eq:rkcnew} of the RKC method applied to \eqref{eq:P} for which the internal stages behave similarly to that of the order one method, given by
\begin{equation} \label{eq:rkcoc}
	\begin{split}
		\min&~\Psi(y_N),~such~that\\
		y_{k0} &= y_{k},\quad
		y_{k1} = y_{k0}+ \mu_1 h f(u_{k0},y_{k0}),\\
		y_{ki} &= \mu_ihf(u_{k,{i-1}},y_{k,{i-1}})+\nu_iy_{k,{i-1}}+(1-\nu_i)y_{k,{i-2}},\quad i=2,\ldots,s\\
		y_{k+1} &= a_sy_{k0}+b_sT_s(\omega_{0})y_{ks},
	\end{split}
\end{equation}
where, $k=0,\dots,N-1,~ \eta=0.15,$ and again all the parameters are defined as for the Chebyshev method using $\omega_2$ instead of $\omega_{1}$.
The order two RKC method with formulation \eqref{eq:rkcoc} for the state equation can be implemented 
using Algorithm \ref{algo:iter}, the costate equation
being implemented using the recurrence relations given by the following Theorem \ref{th:rkc}. Its proof is analogous to that of Theorem \ref{th:cheb} and thus omitted.
\begin{theorem}
	The double adjoint of the scheme \eqref{eq:rkcoc} is given by the recurrence 
	\begin{equation}\label{eq:rkcda}
		\begin{split}
			p_N &= \nabla\Psi(y_N),\quad p_{ks}=p_{k+1},\\
			p_{k,{s-1}}&=p_{ks}+\frac{\mu_s}{\nu_s}h\nabla_yH(u_{k,{s-1}},y_{k,{s-1}},p_{ks}),\\
			p_{k,{s-j}}&=\frac{\mu_{s-j+1}\alpha_{s-j+1}}{\alpha_{s-j}} h \nabla_yH(u_{k,{s-j}},y_{k,{s-j}},p_{k,{s-j+1}})\\&+\frac{\nu_{s-j+1}\alpha_{s-j+1}}{\alpha_{s-j}}p_{k,{s-j+1}},\\
			&+\frac{(1-\nu_{s-j+2})\alpha_{s-j+2}}{\alpha_{s-j}}p_{k,{s-j+2}},\quad j=2,\dots,s-1,\\
			p_{k0}&= \mu_1\alpha_1h\nabla_yH(u_{k0},y_{k0},p_{k1})+\alpha_1p_{k1}+(1-
			\nu_2)\alpha_2p_{k2}+a_sp_{k+1},\\
			p_k&=p_{k0},\\
			&\nabla_uH(u_{k,{s-j}},y_{k,{s-j}},p_{k,{s-j+1}})=0,\quad j=1,\dots,s,
		\end{split}
	\end{equation}
	where the coefficients $\alpha_j$ are defined using the induction
	\begin{equation}\label{eq:alphaind2}
		\begin{split}
			\alpha_s&=b_sT_s(\omega_{0}),\quad
			\alpha_{s-1}=\nu_s\alpha_s,\\
			\alpha_{s-j}&=\nu_{s-j+1}\alpha_{s-j+1}+(1-\nu_{s-j+2})\alpha_{s-j+2},~~j=2\dots s-1.
		\end{split}
	\end{equation}
	\label{th:rkc}
\end{theorem}

In Figure \ref{fig:dastages}, we plot the stability function and the internal stages of the double adjoint \eqref{eq:chebda} of Chebyshev \eqref{eq:Cheb1} and the double adjoint \eqref{eq:rkcda} of RKC \eqref{eq:rkcnew}. Comparing with Figures \ref{fig:stagescheb}b and \ref{fig:stagesrkc}b, we observe that the internal stages are not the same for the double adjoint methods compared to the \eqref{eq:Cheb1} and \eqref{eq:rkcnew}, while the stability function itself is identical as shown in Theorem \ref{th:stabfun} for a general Runge-Kutta method.
\begin{figure}
	\centering
	\begin{subfigure}{0.49\linewidth}
		\includegraphics[width=\linewidth]{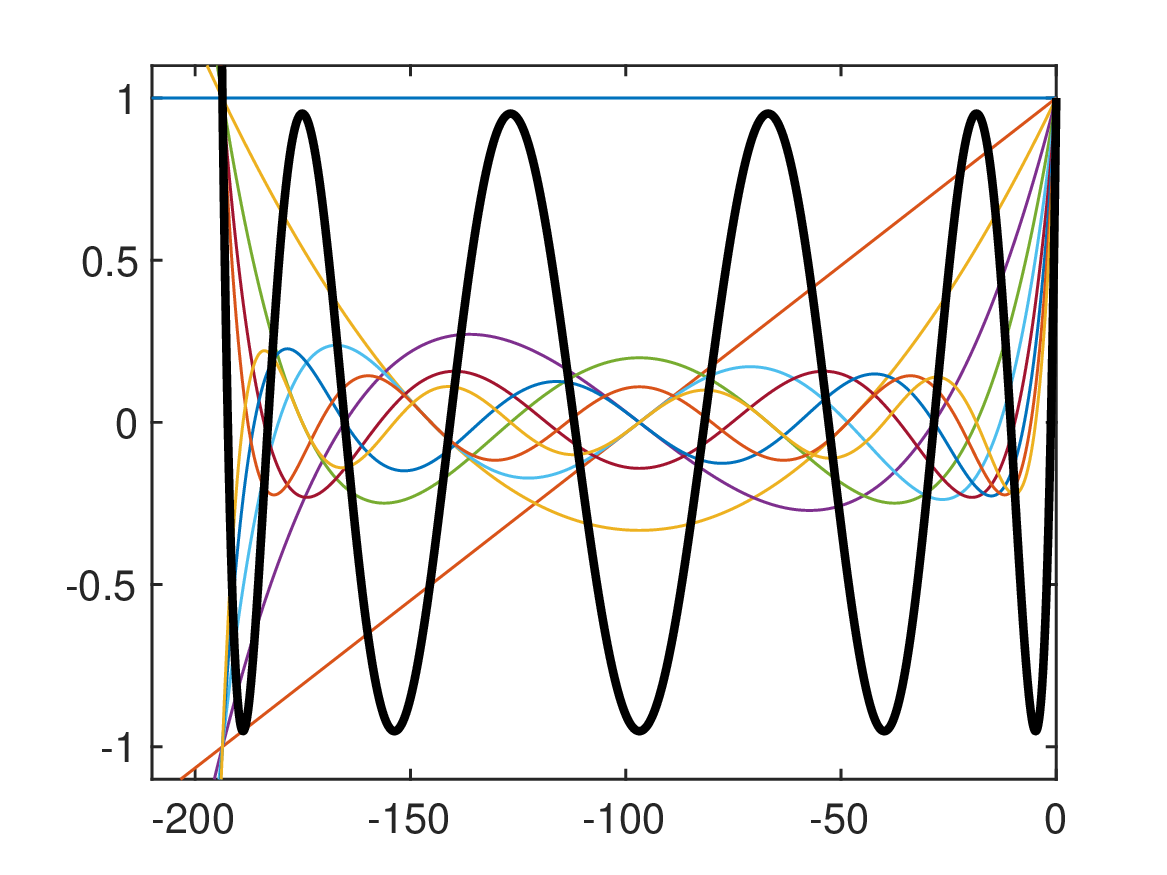}
		\caption{Double adjoint \eqref{eq:chebda} for $\eta=0.05$}
		\label{fig:stagechebda}
	\end{subfigure}
	\begin{subfigure}{0.49\linewidth}
		\includegraphics[width=\linewidth]{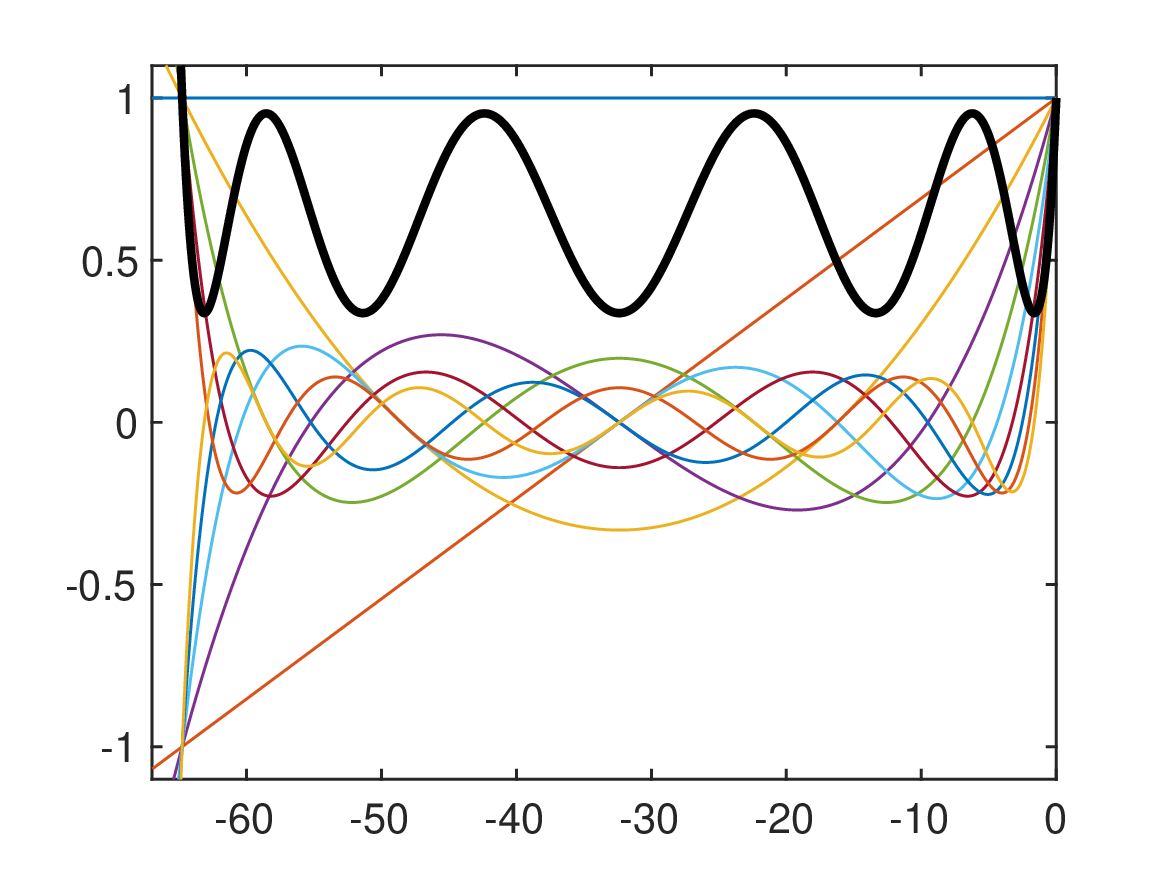}
		\caption{Double adjoint \eqref{eq:rkcda} for $\eta=0.15$}
		\label{fig:rkcdastages}
	\end{subfigure}
	\caption{Internal stages (thin curves) and stability polynomial (bold curve) of the double adjoint of the Chebyshev method \eqref{eq:chebda} of order one and the RKC method \eqref{eq:rkcda} of order two for $s=10$ internal stages.}
	\label{fig:dastages}
\end{figure}

\subsection{Stability and convergence analysis}\label{sec:stability}

In this section, we study the stability of the double adjoint of the Chebyshev method \eqref{eq:cheboc} and the RKC method \eqref{eq:rkcoc}. We recall that for the Chebyshev method of order one (resp. RKC of order two) the stability domain contains the interval $[-\beta(s,\eta),0]$  where $\beta_{Cheb}(s,\eta) \approx (2-4\eta/3)s^2$ (resp. $\beta_{RKC}(s,\eta)\approx 0.653s^2$ for $\eta = 0.15$).

\begin{theorem}\label{th:stab}
	Consider the Chebyshev \eqref{eq:cheboc} and the RKC \eqref{eq:rkcoc} methods. For $\eta=0$, the stability functions of the internal stages $R_{s,i}(z)$ of the Chebyshev (resp. RKC) double adjoint \eqref{eq:chebda} (resp. \eqref{eq:rkcda}), are bounded by $1$ for all $z\in[-2s^2,0]$ (resp. $[-\frac23 s^2+\frac23,0]$) and all $s\in \N$.
\end{theorem}
The proof of Theorem \ref{th:stab} relies on the following lemma.
\begin{lemma}\label{indgamma}
	Let $s\geq 1$, and consider the double sequence 
	$\tilde \gamma^i_j$ 
	indexed by $i$ and $j$,
	\begin{equation}\label{eq:initial}
		\begin{split}
			\tilde \gamma^i_j&=0 ~\forall~ j>i, \qquad i=0,\dots,s-1,\\ 
			\tilde \gamma^0_0&=1,\quad \tilde \gamma^1_0=0,\quad \tilde \gamma^1_1=2,\\
		\end{split}
		~~~~~~~~~~~~~~~~~~~~~~~~~~~
	\end{equation}
	\begin{equation}\label{eq:induction}
		\begin{split}
			\tilde \gamma^i_0&=\tilde \gamma^{i-1}_1-\tilde \gamma^{i-2}_0,\quad
			\tilde \gamma^i_1=2\tilde \gamma^{i-1}_0+\tilde \gamma^{i-1}_2-\tilde \gamma^{i-2}_1\qquad i=2,\dots,s-1,\\
			\tilde \gamma^i_j&=\tilde \gamma^{i-1}_{j-1}+\tilde \gamma^{i-1}_{j+1}-\tilde \gamma^{i-2}_j\qquad i=2,\dots,s-1,~j=2,\dots,i,\\		
		\end{split}
	\end{equation}
	Then,
	\begin{equation}\label{eq:expcoeff}
		\begin{split}
			\tilde \gamma^i_j&=0 ~\forall~ j>i,\\
			\tilde \gamma^i_0 &= \begin{cases}
				1\quad \text{if $i$ is even}\\
				0\quad \text{if $i$ is odd}
			\end{cases}\quad 
			\tilde \gamma^i_j=\begin{cases}
				2\quad \text{if $i-j$ is even}\\
				0\quad \text{otherwise}
			\end{cases}
		\end{split}
	\end{equation}
	where  $i=0,\dots,s-1,~ j=0,\dots,i$.
\end{lemma}
\begin{proof}
	It can be checked that the coefficients defined in \eqref{eq:expcoeff} verify the induction \eqref{eq:induction}. Hence using the fact that they have the same initial terms \eqref{eq:initial}, we conclude that they coincide by induction on $i$ and $j$.
\end{proof}
\begin{proof}[Proof of Theorem \ref{th:stab}]
	
	We first consider the Chebyshev method without damping applied to the linear test problem $y'=\lambda y~,\lambda\in \C,~t\in (0,T],~y(0)=1$, with a uniform subdivision $x_0=0<x_1<\dots<x_N=T$ of stepsize $h$.
	Using Remark \ref{rem:nodamp}, we obtain for $k=0$:
	\begin{equation}\label{eq:chebda0proof}
		\begin{split}
			y_{k0}&=1,\quad
			y_{k1}=y_{k0}+\frac{h\lambda}{s^2}y_{k0},\\
			y_{ki}&=\frac{2ih\lambda}{(i+1)s^2}y_{k,{i-1}}+\frac{2i}{i+1}y_{k,{i-1}}+\frac{1-i}{i+1}y_{k,{i-2}},\quad i=2,\ldots,s-1,\\
			y_{ks}&=\frac{h\lambda}sy_{k,{s-1}}+sy_{k,{s-1}}+(1-s)y_{k,{s-2}},\\
			y_1&=y_{ks}.\\
		\end{split}
	\end{equation}
	First, notice that since $\frac{2i}{i+1}+\frac{1-i}{i+1}=1$, we have that for all $i=0,\dots,s$, $y_{ki}=(1+\mathcal{O}(h))$ (by induction). Setting $z=h\lambda$, it is sufficient to prove the identity
	\begin{equation}\label{eq:sum}
		y_{ki}=\sum_{j=0}^{i} \gamma^i_jT_{j}(1+\frac z{s^2})
	\end{equation}
	where $y_{ki}$ is a convex combination of the polynomials $T_{j}(1+\frac z{s^2})$,
	\begin{equation}\label{eq:sum1}
		\sum_{j=0}^{i}\gamma^i_j=1 \text{ and } \gamma^i_j\geq0~ \forall~ i,j=1,\dots s-1,
	\end{equation}
	because $|T_j(1+\frac z{s^2})|\leq 1$ for all $j=0,\dots,s$ and $z\in[-\beta(s,0),0]=[-2s^2,0]$.\\
	Since the Chebyshev polynomials form a basis of the vector space of polynomials, this already justifies the existence of the expansion \eqref{eq:sum} with some real coefficients $\gamma^i_j$. The identity $\sum_{j=0}^{i}\gamma^i_j=1$ follows from the fact that $y_{ki}=1+\mathcal{O}(h)$ and $T_j(1+\frac{z}{s^2})=1+\mathcal{O}(h)$ for all $i,j=0,\dots,s$.
	Now we can calculate these coefficients for the first two internal stages,
	$R_{s,0}(z)=y_{k0}=y_k=1=T_0(1+\frac z{s^2})$, thus $\gamma^0_0=1$.
	Analogously, $R_{s,1}(z)=y_{k1}=y_{k0}+\frac{h\lambda}{s^2}y_{k0}=1+\frac{z}{s^2}=T_1(1+\frac z{s^2})\text{ we obtain }\gamma^1_0=0,~\gamma_1^1=1.$\\
	It remains to prove the positivity of the coefficients $\gamma^i_j$. Coupling \eqref{eq:sum} and \eqref{eq:sum1}, we obtain
	\begin{equation*}
		\begin{split}
			R_{s,i}(z)=y_{ki}
			&=\frac{2i}{i+1}(1+\frac z{s^2})y_{k,{i-1}}+\frac{1-i}{i+1}y_{k,{i-2}}\\
			&=\frac{2i}{i+1}(1+\frac z{s^2})\sum_{j=0}^{i-1} \gamma^{i-1}_jT_{j}(1+\frac z{s^2}) + \frac{1-i}{i+1}\sum_{j=0}^{i-2} \gamma^{i-2}_jT_{j}(1+\frac z{s^2})\\
			&=\frac {2i}{i+1}\gamma^{i-1}_0T_1(1+\frac z{s^2}) + 
			\sum_{j=2}^{i}\frac i{i+1}\gamma^{i-1}_{j-1}T_{j}(1+\frac z{s^2})\\
			&+\sum_{j=2}^{i}\left(\frac i{i+1}\gamma^{i-1}_{j-1}-\frac{i-1}{i+1}\gamma^{i-2}_{j-2}\right)T_{j-2}(1+\frac z{s^2})
		\end{split}
	\end{equation*}
	where we used \eqref{eq:recTU1}.
	By comparison with \eqref{eq:sum}, we obtain the following recurrence
	\begin{equation*}
		\begin{split}
			\gamma^i_j&=0 ~\forall~ j>i, \qquad i=0,\dots,s-1,\\
			\gamma^0_0&=1,\quad \gamma^1_0=0,\quad \gamma^1_1=1,\quad
			\gamma^i_0=\frac i{i+1}\gamma^{i-1}_1-\frac{i-1}{i+1}\gamma^{i-2}_0 \qquad i=2,\dots,s-1,\\
			\gamma^i_1&=\frac{2i}{i+1}\gamma^{i-1}_0+\frac i{i+1}\gamma^{i-1}_2-\frac{i-1}{i+1}\gamma^{i-2}_1\qquad i=2,\dots,s-1,\\
			\gamma^i_j&=\frac{i}{i+1}\gamma^{i-1}_{j-1}+\frac i{i+1}\gamma^{i-1}_{j+1}-\frac{i-1}{i+1}\gamma^{i-2}_j\qquad i=2,\dots,s-1,~j=2,\dots,i.\\		
		\end{split}
	\end{equation*}
	Now defining $\tilde{\gamma}^i_j=(i+1)\gamma^i_j$, the above induction relations 
	simplify to \eqref{eq:initial} and \eqref{eq:induction}.
	The positivity of $\tilde{\gamma}^i_j$, and hence of $\gamma^i_j$, follows from Lemma \ref{indgamma}. For $i=s$, the stability is a consequence of Theorem \ref{th:stabfun}.
	
	Analogously, the RKC method reads for $\eta=0$,
	\begin{equation}
		\begin{split}
			y_{k0}&=1,\quad
			y_{k1}=y_{k0}+\frac{3h\lambda}{s^2-1}y_{k0},\\
			y_{ki}&=\frac{6ih\lambda}{(i+1)(s^2-1)}y_{k,{i-1}}+\frac{2i}{i+1}y_{k,{i-1}}+\frac{1-i}{i+1}y_{k,{i-2}},\quad i=2,\ldots,s-1,\\
			y_{ks}&=\frac{h\lambda}sy_{k,{s-1}}+\frac{s^2-1}{3s}y_{k,{s-1}}-\frac{s^3-s^2-s+1}{3s^2}y_{k,{s-2}}+\frac{2s^2+1}{3s^2}y_0,\\
			y_1&=y_{ks},\\
		\end{split}
	\end{equation}
	and we follow the same methodology as above. Note that for RKC we have $|T_j(1+\frac 3{s^2-1}z)|\leq 1$ for all $j=0,\dots,s$ and $z\in[-\beta(s,0),0]=[-\frac2{3}(s^2-1),0]$. Using the same notations we search for coefficients satisfying the following
	\begin{equation}\label{eq:ind2}
		y_{ki}=\sum_{j=0}^{i} \gamma^i_jT_{j}(1+\frac {3}{s^2-1}z),~\text{where}~\sum_{j=0}^{i}\gamma^i_j=1 \text{ and } \gamma^i_j\geq0~ \forall~ i,j=1,\dots s-1.
	\end{equation}
	Remark that $R_{s,0}(z)=y_{k0}=y_k=1=T_0(1+\frac 3{s^2-1}z)$, hence $\gamma^0_0=1$.
	Analogously, $R_{s,1}(z)=y_{k1}=y_{k0}+\frac{3h\lambda}{s^2-1}y_{k0}=1+\frac{3z}{s^2-1}=T_1(1+\frac 3{s^2-1}z)$, and we deduce that $\gamma^1_0=0,~\gamma_1^1=1.$\\
	Again we find a relation between these new coefficients to prove their positivity using
	$
	y_{ki}=\frac{2i}{i+1}(1+\frac 3{s^2-1}z)y_{k,{i-1}}+\frac{1-i}{i+1}y_{k,{i-2}}.		
	$
	We get a recurrence of the same form as in the Chebyshev double adjoint method \eqref{eq:chebda0proof} but with different parameter, proceeding in the same we obtain exactly the same coefficients $\gamma^i_j$, which concludes the proof.
\end{proof}
\begin{remark}
	For the case of positive damping $\eta>0$, the coefficients get very complicated and it is difficult to find a recurrence relation between them in order to prove their positivity. However, observing that all the coefficients in the recurrence relations of the internal stages of the methods are continuous functions of $\eta$, then for all $s$, there exists $\eta_0(s)$ such that the internal stages are stable (bounded) for all $\eta\in[0,\eta_0(s)]$.
	Numerical investigations suggest that Theorem \ref{th:stab} remains valid for all $\eta>0$ i.e the methods remain stable with the stability functions of the internal stages bounded by $1$, for all integers $s\geq1$ for Chebyshev and $s\geq2$ for RKC, and all $\eta>0$. We have verified this numerically for $ s\leq 200$.
\end{remark}

We conclude this section by the following convergence theorem for the new explicit stabilized methods for stiff optimal control problems.
\begin{theorem}\label{th:conv}
	The method \eqref{eq:cheboc}-\eqref{eq:chebda} (resp. \eqref{eq:rkcoc}-\eqref{eq:rkcda}) has order 1 (resp. 2) of accuracy for the optimal control problem \eqref{eq:P}.
\end{theorem}
\begin{proof}
	The proof follows immediately from Theorem \ref{th:diagram} with $p_{oc}=p_{ode}=2$ for the RKC method.
\end{proof}

\begin{remark}
	The proposed explicit stabilized integrators for optimal control pro\-blems could be combined with the idea of implicit-explicit (IMEX) integrators as proposed in \cite{HPS13}, where RKC type methods would replace the implicit part in the IMEX integrator. This idea is already proposed in \cite{Zb11,AV13} in the context of  advection-diffusion-reaction problems. In \cite{Zb11}, the diffusion part is  discretized  with an RKC method which typically has a large number of internal stages, and the advection-reaction part is integrated using a 4-stage explicit Runge-Kutta method. In \cite{AV13}, the method integrates the diffusion term using ROCK2 method, the advection term using a 3-stage explicit method, and the nonlinear reaction term is solved implicitly. Such an extension is however out of the scope of the paper.
\end{remark}

\section{Numerical experiments}
\label{sec:num}

In this Section, we illustrate numerically our theoretical findings of convergence and stability of the new fully explicit methods for stiff optimal control problems, first on a stiff three dimensional problem and second on a nonlinear diffusion advection PDE (Burgers equation).
\subsection{A linear quadratic stiff test problem}
We start this section by a simple test problem taken from \cite{Hager00}:
\begin{equation}\label{eq:hager}
	\begin{split}
		\min~ &\frac{1}{2} \int_{0}^{1}(u^2(t)+2x^2(t))dt \text{ subject to}\\
		&\dot x(t)=\frac12 x(t)+u(t),~t\in[0,1],~x(0)=1.
	\end{split}
\end{equation}
The optimal solution $(u^*,x^*)$ is given by
$
u^*(t)=\frac{2(e^{3t}-e^3)}{e^{3t/2}(2+e^3)}, \ts x^*(t)=\frac{2e^{3t}+e^3}{e^{3t/2}(2+e^3)}
$.
As studied in \cite{HPS13} we modify problem \eqref{eq:hager} into a singularly perturbed (stiff) problem to illustrate the good stability properties of our new method. For a fixed $\epsilon>0$, we consider the following stiff optimal control problem,
\begin{equation}\label{eq:imex44}
	\begin{split}
		\min~ &c(1) \text{ subject to}\\
		&\dot c(t)=\frac12 (u^2(t)+x^2(t)+4z^2(t)),~c(0)=0,\\
		&\dot x(t)=z(t)+u(t),~x(0)=1,\\
		&\dot z(t)=\frac1\epsilon \left(\frac12x(t)-z(t)\right),~z(0)=\frac12,
	\end{split}
\end{equation}

%

\begin{figure}[t]
	\centering
	\begin{subfigure}{0.49\linewidth}
		\includegraphics[width=\linewidth]{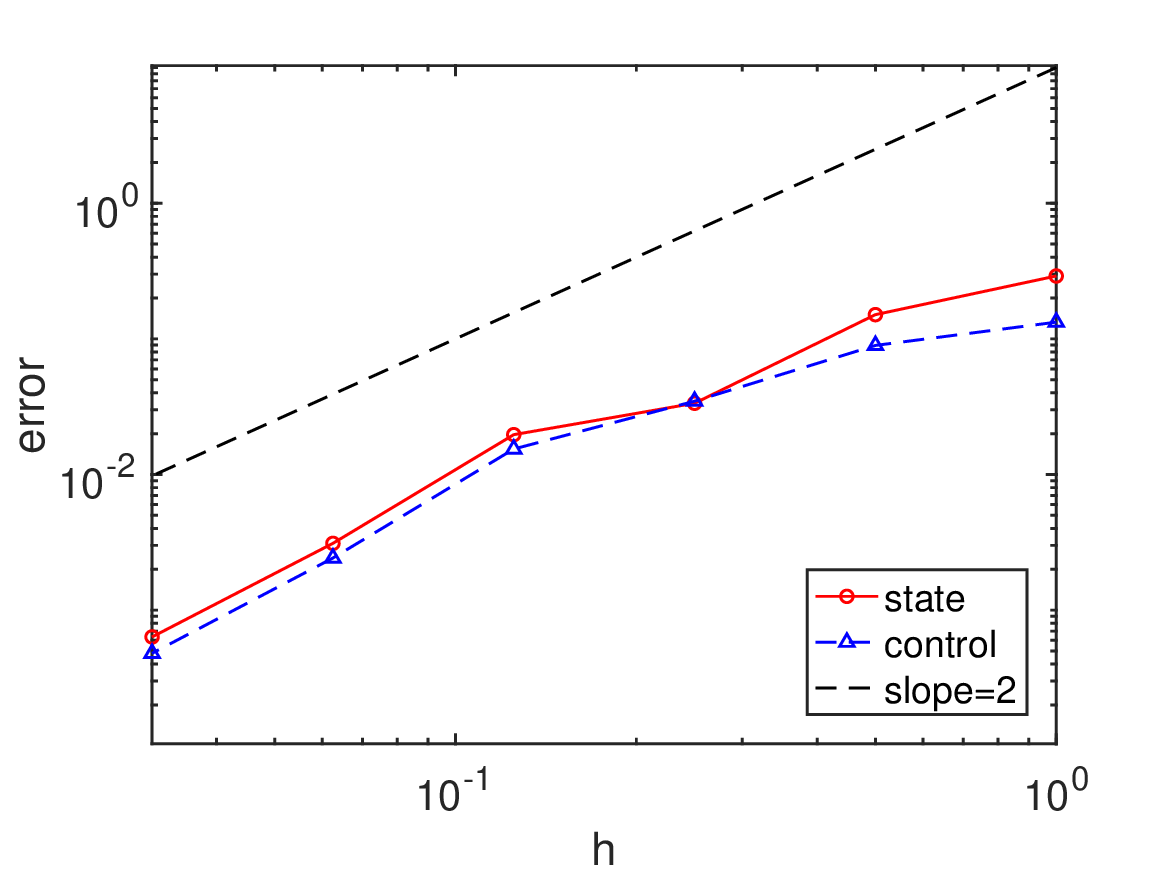}
		\caption{$\epsilon=10^{-1}$.}
	\end{subfigure}
	\begin{subfigure}{0.49\linewidth}
		\includegraphics[width=\linewidth]{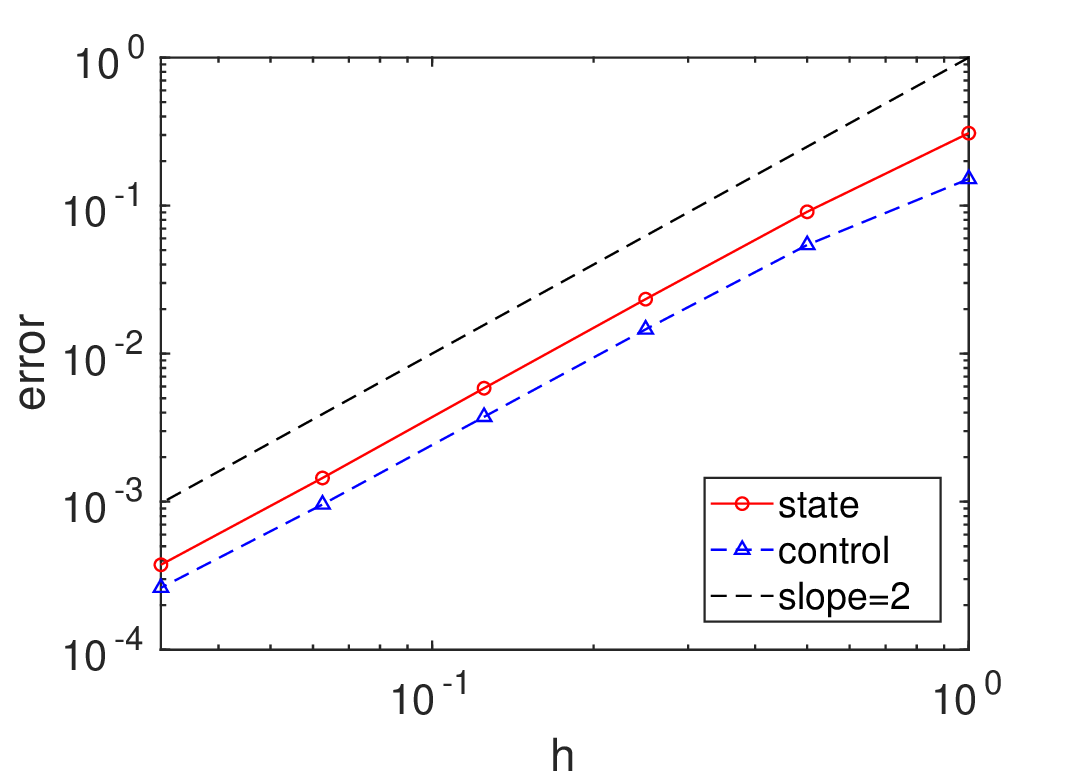}
		\caption{$\epsilon=10^{-3}$.}
	\end{subfigure}
	\caption{Convergence plot of RKC \eqref{eq:rkcoc}-\eqref{eq:rkcda} applied to problem \eqref{eq:imex44}.}
	\label{fig:plotconvpb1}
\end{figure}

Figure \ref{fig:plotconvpb1} shows the convergence behavior, using 
the new RKC method \eqref{eq:rkcoc}-\eqref{eq:rkcda}, of the error in infinity norm between the solutions of the stiff problem \eqref{eq:imex44} for $\eps=10^{-1}$ and $\eps=10^{-3}$ and different sizes of the time step $h_i=2^{-i}, i=0,\dots,5$ and the reference solution is obtained with $h=2^{-7}$. We observe lines of slope $2$ which confirms the theoretical order two of accuracy of the scheme (Theorem \ref{th:conv}). In the stiff case $(\eps=10^{-3})$, the method uses $s=4$ to calculate the reference solution and $s=40, 28, 20, 14, 10, 7$ respectively for the different time steps used to illustrate the convergence,  these values coincide with the theoretical values that can be obtained using \eqref{eq:defsnumrkc}. Analogously to the case of stiff ODEs, the cost of scheme \eqref{eq:rkcoc}-\eqref{eq:rkcda} is $\bigo(\eps^{-\frac12})$ function evaluations of $f$, while using Euler method with its double adjoint would cost $\bigo(\eps^{-1})$.

\subsection{Optimal control of Burgers equation}
\label{sec:burger}
To illustrate the performance of the new method, we consider the following optimal control problem of a nonlinear diffusion advection PDE corresponding to the Burgers equation
\begin{equation}\label{eq:burger}
	\begin{split}
		\min_{u\in L^2([0,T];L^2(\Omega))} &J(u)=\frac12 \|y(T)-y^{target}\|_{L^2(\Omega)}^2 + \frac{\alpha}{2}\int_{0}^{T}\|u(t)\|_{L^2(\Omega)}^2\\
		&\text{subject to}\\
		&\partial_t y(t,x)=\mu\Delta y(t,x)-\frac\nu2\partial_x(y^2(t,x))+u(t,x)\quad \text{in } (0,T)\times\Omega,\\
		&y(0,x)=g(x)\quad\text{in } \Omega,\\
		&y(t,x)=0\quad \text{on } \partial \Omega,
	\end{split}
\end{equation}
where $\mu,\nu>0$, in dimension $d=1$ with domain $\Omega=(0,1)$ and the final time is given by $T=2.5$. Here the control $u$ is a part of the source that we want to adjust in order to achieve a given final state $y^{target}:\Omega\to\IR$.

We use a standard central finite difference space discretization for the state equation, and the trapezoid rule to discretize in space the norm $L^2(\Omega)$. We consider $M+2$ points in space $x_m=m\Delta x$, with grid mesh size $\Delta x=\frac1{M+1}$, and we denote by $y_m(t)$ the approximation to $y(t,x_m)$, and define the vector $Y(t)=(y_0(t),y_1(t),\dots,y_{M+1}(t))\in \IR^{M+2}$. Similar notations are used for $U$ and $P$. We obtain the following  optimal control problem semi discretized in space,
\begin{equation}\label{eq:PDEdiscr}
	\begin{split}
		\min_{U(t)\in\IR^{M+2}} \quad &\Psi(c(T),Y(T))=\frac1{2(M+1)} \sum_{m=0}^{M+1}{}^{'}(y_m(T)-y^{target}(x_m))^2+\alpha c(T)\\
		&\text{subject to}\\
		\dot c(t)&=\frac1{2(M+1)} \sum_{m=0}^{M+1}{}^{'}u_m^2(t),\quad c(0)=0,\\
		\dot y_m(t)&=F_m(U(t),Y(t)):=\frac{\mu}{\Delta x^2}(y_{m+1}-2y_m+y_{m-1})\\&\phantom{=F_m(U(t),Y(t)):}-\frac\nu{4\Delta x^2}(y^2_{m+1}-y^2_{m-1})+u_m,\\ 
		y_m(0)&=g(x_m),~m=0,\dots,M+1,
	\end{split}
\end{equation}
where $m=0,\dots M+1$ and the primed sum denotes a normal sum where the first and the last term are divided by $2$ and we define $y_0=y_{M+1}=0$ to take into account the homogeneous Dirichlet boundary conditions. The function $F:\IR^{M+2}\to\R^{M+2}$ with components $F_m:\R^{M+2}\to\R$ is obtained from the standard central finite difference discretization of the right hand side of the state equation \eqref{eq:burger}, and adapted to the boundary conditions.
The corresponding adjoint system is
\begin{equation}
	\begin{split}
		\dot p_c(t)&=0,\quad p_c(T)=\nabla_c\Psi=\alpha,\\
		\dot P(t)&=-\nabla_YF(U(t),Y(t))P,
	\end{split}
\end{equation} 
where 
$P$ is a vector of length $M+2$ containing the costate values $p_m$, $m=0,\dots, M+1$. In all our experiments we take $\mu=0.1$, $\nu=0.02$, $g(x)=\frac32x(1-x)^2$, and $y^{target}(x)=\frac12\sin(10x)(1-x)$.

In Figure \ref{fig:burger} we plot the optimal control function (Fig.\ts\ref{fig:burger}b) and the corresponding state function (Fig.\ts\ref{fig:burger}a) obtained using scheme \eqref{eq:rkcoc}-\eqref{eq:rkcda}. When we use a small value for $\alpha$ in the model, we allow larger control values and thus a final state very close to the target (Fig.\ts\ref{fig:burger}c), otherwise the control will be more limited and then the final state will not be that close to the target (Fig.\ts\ref{fig:burger}d). Note that the method required $s=24$ stages for $\Delta x=1/100$ and $\Delta t=2.5/30$, while using an Euler method with its double adjoint would require $\Delta t\leq \Delta t_{max,Euler} :=\Delta x^2/2$ at most. 
Hence, the standard Euler method would be $\Delta t/(s\Delta t_{max,Euler})\simeq 70$ times more expensive in terms of number of function evaluations for $\Delta x=1/100$, a factor that grows arbitrarily as $\Delta x \rightarrow 0$.
\begin{figure}[t]
	\centering
	\begin{subfigure}{0.49\linewidth}
		\includegraphics[width=\linewidth]{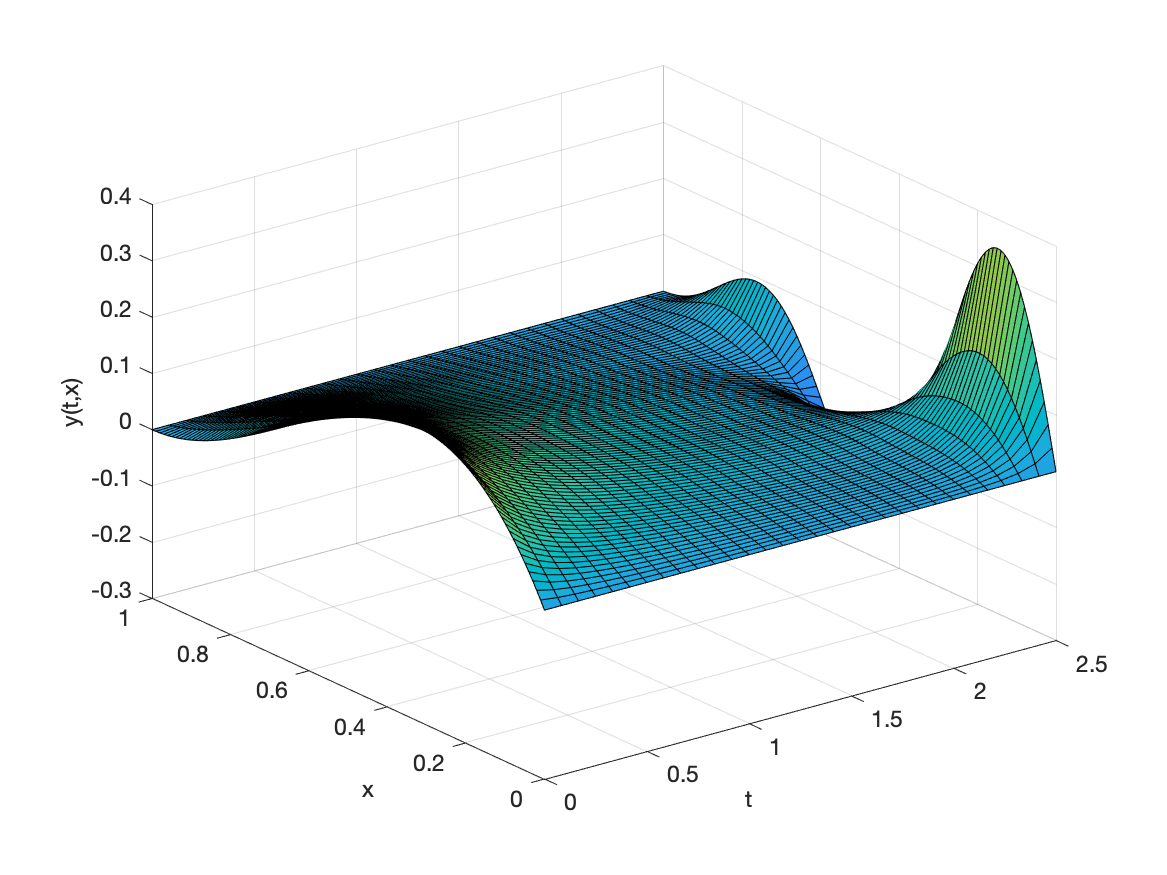}
		\caption{State.}
		\label{fig:state}
	\end{subfigure}
	\begin{subfigure}{0.49\linewidth}
		\includegraphics[width=\linewidth]{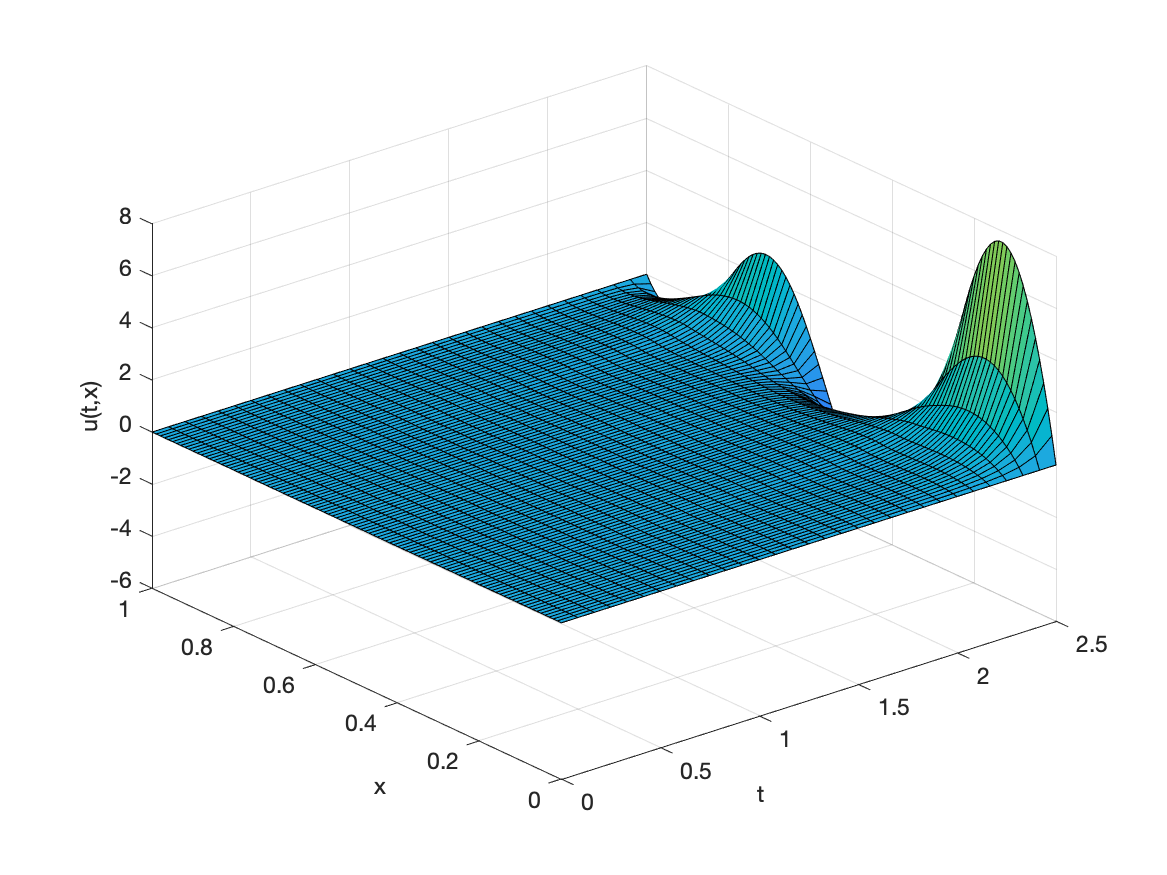}
		\caption{Optimal control}
		\label{fig:control}
	\end{subfigure}\\
	\begin{subfigure}{0.49\linewidth}
		\includegraphics[width=\linewidth]{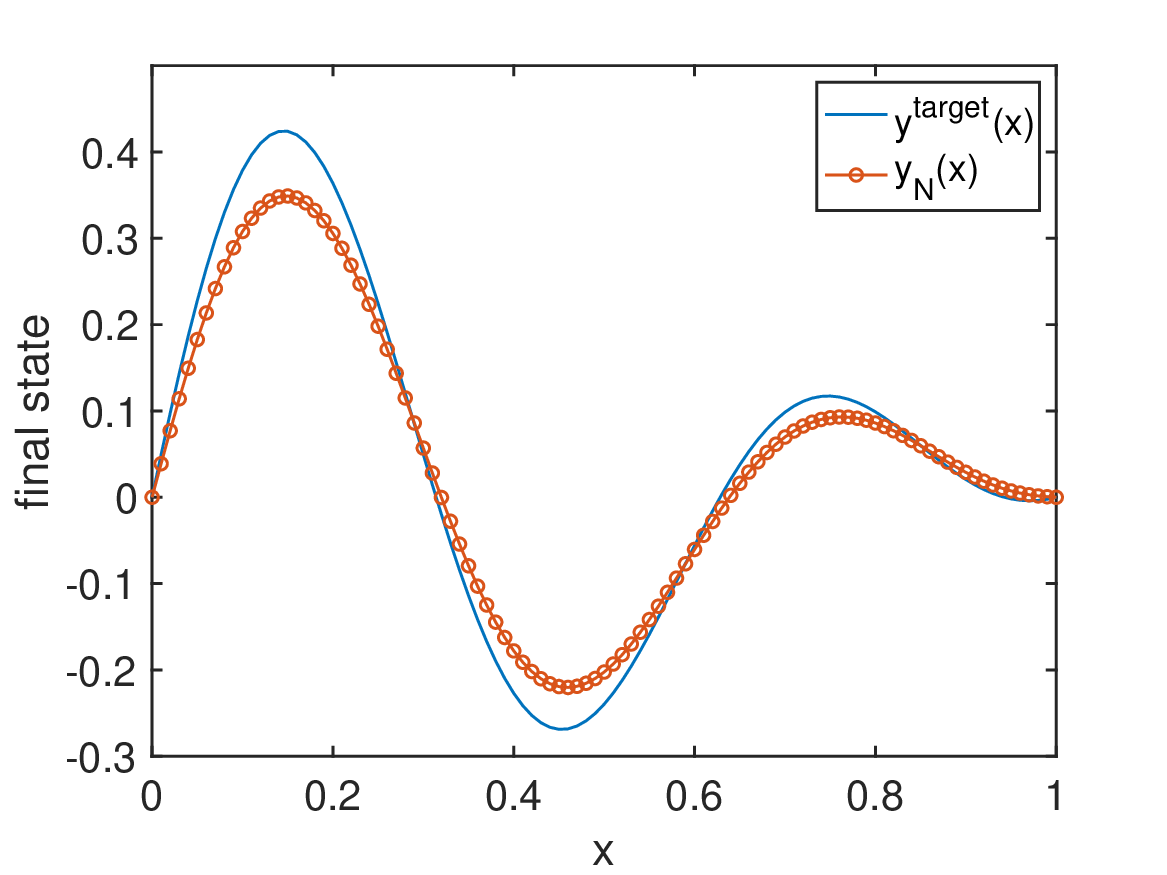}
		\caption{Final and target states for $\alpha=0.01$.}
		\label{fig:target1}
	\end{subfigure}
	\begin{subfigure}{0.49\linewidth}
		\includegraphics[width=\linewidth]{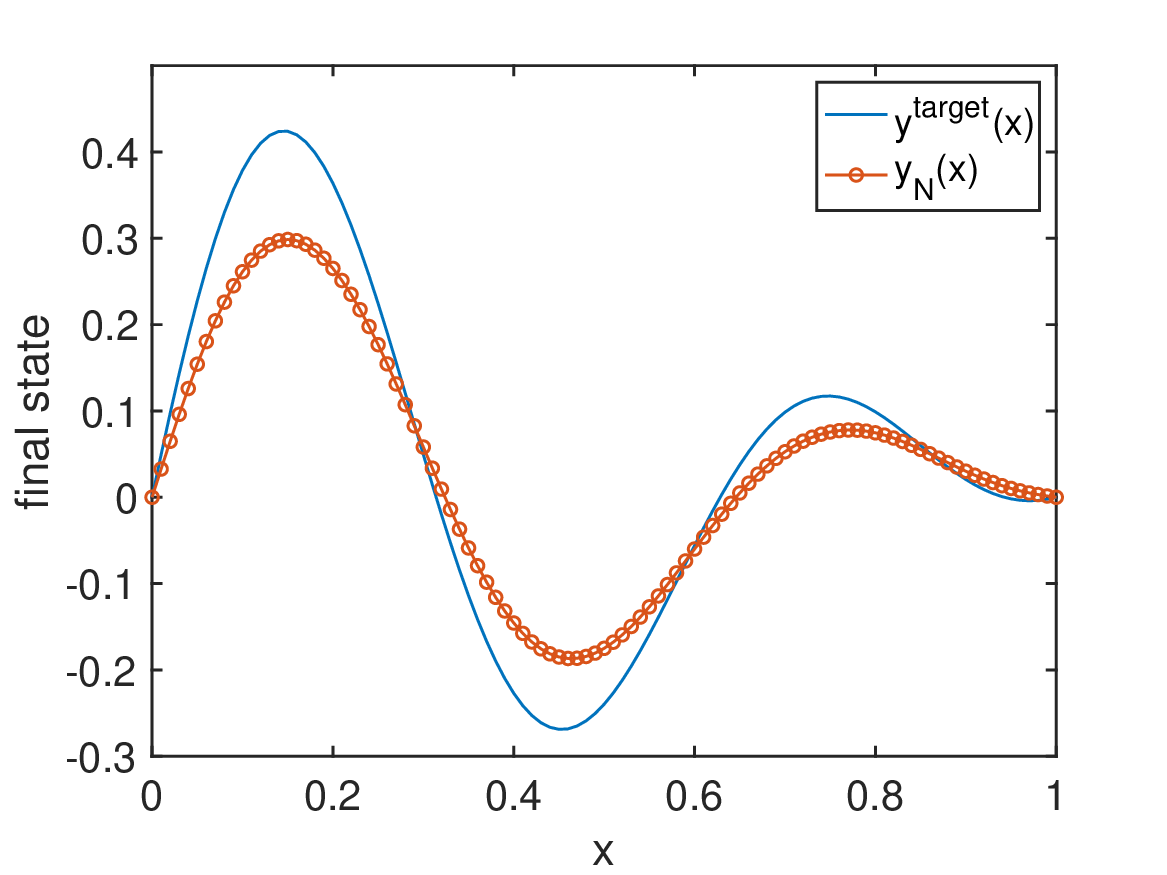}
		\caption{Final and target states for $\alpha=0.02$.}
		\label{fig:conv2}
	\end{subfigure}
	\caption{State, final state, and control, of problem \eqref{eq:burger}. Figures (a), (b), and (c) are obtained using $\Delta x=1/100,~ \Delta t=T/30,$ and $s=24$ stages,  and $\alpha=0.01$. Figure (d) uses the same $\Delta x,~g$ and  $y^{target}$ but $\alpha=0.02$.}
	\label{fig:burger}
\end{figure}
\begin{figure}[t]
	\centering
	\begin{subfigure}{0.49\linewidth}
		\includegraphics[width=\linewidth]{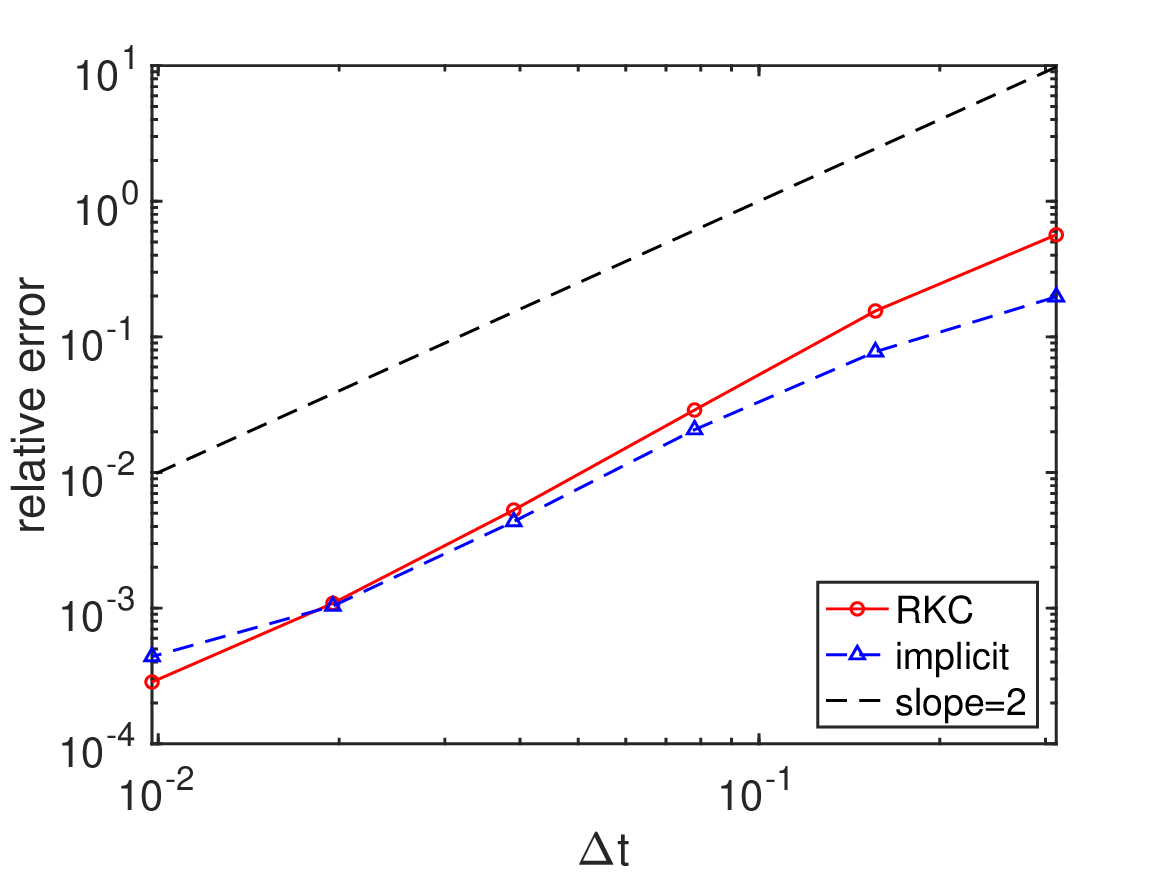}
		\caption{Error in the state.}
		\label{fig:convstate}
	\end{subfigure}
	\begin{subfigure}{0.49\linewidth}
		\includegraphics[width=\linewidth]{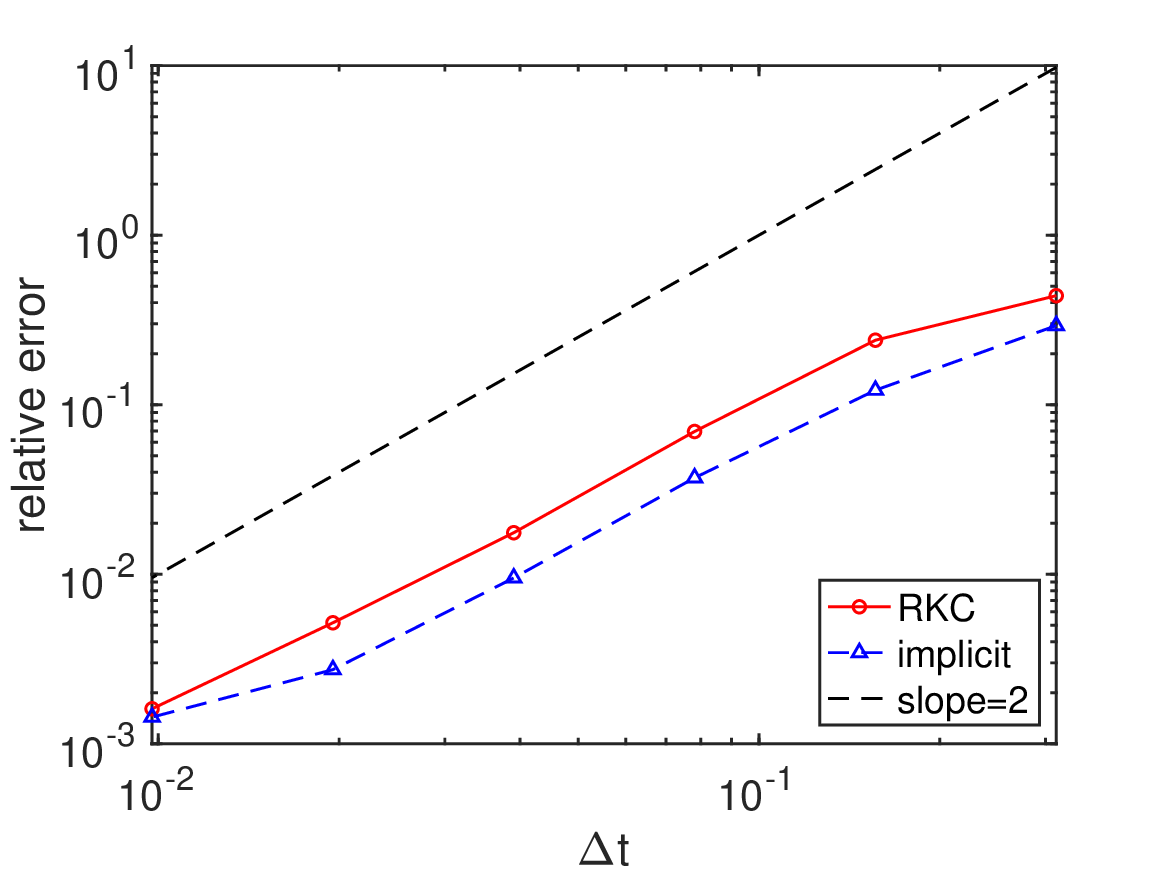}
		\caption{Error in the control.}
		\label{fig:convcontrol}
	\end{subfigure}
	\caption{Convergence plot of the
		RKC method \eqref{eq:rkcoc}-\eqref{eq:rkcda} and the implicit method \eqref{eq:imp} applied to problem \eqref{eq:PDEdiscr}
		for many time steps $\Delta t_i=T/2^{i},~i=3,\dots,8$, $\Delta x=1/100$, and $\alpha=0.02$, 
		The reference solution is obtained using $\Delta t=T/2^{12},~s=3$.}
	\label{fig:burgerconv}
\end{figure}

In Figure \ref{fig:burgerconv}, we plot the convergence curves for the state and control functions of the new RKC method \eqref{eq:rkcoc}-\eqref{eq:rkcda}
applied to the diffusion problem discretized in space \eqref{eq:PDEdiscr}, where the number of stages $s$ is computed adaptively using \eqref{eq:defsnumrkc}. We recover again lines of slope two, which corroborate the 
order two of the method. Although our convergence analysis is valid only in finite dimensions (Theorem \ref{th:conv}), this suggests that the convergence of order two persists in the PDE case.
For comparison, we also included the results for the following standard diagonally implicit Runge-Kutta method of order two, inspired from \cite[Table 5.1]{HPS13} in the context of stiff optimal control problems, and given by the following Butcher tableau where $\gamma=1-\sqrt2/2$ (making the method L-stable),
\begin{equation}\label{eq:imp}
	\begin{array}
		{c|cc}
		&\gamma \\
		& 1-2\gamma & \gamma \\
		\hline
		& 1/2 & 1/2
	\end{array}
\end{equation}
Although for a fixed timestep, the second order implicit method \eqref{eq:imp} appears about two times more accurate than the RKC method  \eqref{eq:rkcoc}-\eqref{eq:rkcda} for the control and almost of the same accuracy for the state, 
we emphasize that these convergence plots do not take into account the extra cost of the implicitness of method \eqref{eq:imp}. 
Indeed, the cost and difficulty of the implementation of the implicit methods (nonlinear iterations, preconditioners, etc.)
would typically deteriorate in larger dimensions and for a nonlinear diffusion operator, as it is already the case for initial value PDEs \cite{Abd13c}, while~as an explicit stabilized scheme, the RKC method \eqref{eq:rkcoc}-\eqref{eq:rkcda} can be conveniently implemented in the spirit of the simplest explicit Euler method.\\

\noindent \textbf{Acknowledgment.} This work was partially supported by the Swiss National Science Foundation, grants No.200020\_184614, No.200021\_162404 and No.200020\_178752.

\bibliographystyle{abbrv}
\bibliography{HLW,complete,abd_biblio}

\def\cprime{$'$} \def\cprime{$'$} \def\cprime{$'$}
\begin{thebibliography}{10}

\bibitem{Abd02}
A.~Abdulle.
\newblock Fourth order {C}hebyshev methods with recurrence relation.
\newblock {\em SIAM J. Sci. Comput.}, 23(6):2041--2054, 2002.

\bibitem{Abd13c}
A.~Abdulle.
\newblock {\em Explicit Stabilized Runge--Kutta Methods}, pages 460--468.
\newblock Encyclopedia of Applied and Computational Mathematics, Springer
  Berlin Heidelberg, 2015.

\bibitem{AAV18}
A.~Abdulle, I.~Almuslimani, and G.~Vilmart.
\newblock Optimal explicit stabilized integrator of weak order 1 for stiff and
  ergodic stochastic differential equations.
\newblock {\em SIAM/ASA J. Uncertain. Quantif.}, 6(2):937--964, 2018.

\bibitem{AbC08}
A.~Abdulle and S.~Cirilli.
\newblock S-{ROCK}: {C}hebyshev methods for stiff stochastic differential
  equations.
\newblock {\em SIAM J. Sci. Comput.}, 30(2):997--1014, 2008.

\bibitem{AbL08}
A.~Abdulle and T.~Li.
\newblock {S-ROCK} methods for stiff {I}to {SDEs}.
\newblock {\em Commun. Math. Sci.}, 6(4):845--868, 2008.

\bibitem{AbM01}
A.~Abdulle and A.~Medovikov.
\newblock Second order {C}hebyshev methods based on orthogonal polynomials.
\newblock {\em Numer. Math.}, 90(1):1--18, 2001.

\bibitem{AV13}
A.~Abdulle and G.~Vilmart.
\newblock P{IROCK}: a swiss-knife partitioned implicit-explicit orthogonal
  {R}unge-{K}utta {C}hebyshev integrator for stiff diffusion-advection-reaction
  problems with or without noise.
\newblock {\em J. Comput. Phys.}, 242:869--888, 2013.

\bibitem{AHP19}
G.~Albi, M.~Herty, and L.~Pareschi.
\newblock Linear multistep methods for optimal control problems and
  applications to hyperbolic relaxation systems.
\newblock {\em Appl. Math. Comput.}, 354:460--477, 2019.

\bibitem{B71}
M.~Bakker.
\newblock Analytical aspects of a minimax problem.
\newblock 1971.
\newblock Technical Note TN 62 (in Dutch), Mathematical centre, Amsterdam.

\bibitem{BL06}
J.~F. Bonnans and J.~Laurent-Varin.
\newblock Computation of order conditions for symplectic partitioned
  {R}unge-{K}utta schemes with application to optimal control.
\newblock {\em Numer. Math.}, 103(1):1--10, 2006.

\bibitem{Hager00}
W.~W. Hager.
\newblock Runge-{K}utta methods in optimal control and the transformed adjoint
  system.
\newblock {\em Numer. Math.}, 87(2):247--282, 2000.

\bibitem{HLW06}
E.~Hairer, C.~Lubich, and G.~Wanner.
\newblock {\em Geometric numerical integration}, volume~31 of {\em Springer
  Series in Computational Mathematics}.
\newblock Springer-Verlag, Berlin, second edition, 2006.
\newblock Structure-preserving algorithms for ordinary differential equations.

\bibitem{hairer96sod}
E.~Hairer and G.~Wanner.
\newblock {\em Solving {O}rdinary Differential Equations II. Stiff and
  Differential-Algebraic Problems}.
\newblock Springer Series in Computational Mathematics 14. Springer-Verlag,
  Berlin, 2 edition, 1996.

\bibitem{HPS13}
M.~Herty, L.~Pareschi, and S.~Steffensen.
\newblock Implicit-explicit {R}unge-{K}utta schemes for numerical
  discretization of optimal control problems.
\newblock {\em SIAM J. Numer. Anal.}, 51(4):1875--1899, 2013.

\bibitem{HeS11}
M.~Herty and V.~Schleper.
\newblock Time discretizations for numerical optimisation of hyperbolic
  problems.
\newblock {\em Appl. Math. Comput.}, 218(1):183--194, 2011.

\bibitem{HV03}
W.~Hundsdorfer and J.~Verwer.
\newblock {\em Numerical solution of time-dependent
  advection-diffusion-reaction equations}, volume~33 of {\em Springer Series in
  Computational Mathematics}.
\newblock Springer-Verlag, Berlin, 2003.

\bibitem{Kay10}
C.~Y. Kaya.
\newblock Inexact restoration for {R}unge-{K}utta discretization of optimal
  control problems.
\newblock {\em SIAM J. Numer. Anal.}, 48(4):1492--1517, 2010.

\bibitem{LaV13}
J.~Lang and J.~G. Verwer.
\newblock W-methods in optimal control.
\newblock {\em Numer. Math.}, 124(2):337--360, 2013.

\bibitem{LCTE18}
Q.~Li, L.~Chen, C.~Tai, and W.~E.
\newblock Maximum principle based algorithms for deep learning.
\newblock {\em Journal of Machine Learning Research}, 18(165):1--29, 2018.

\bibitem{LF19}
X.~Liu and J.~Frank.
\newblock Symplectic {R}unge-{K}utta discretization of a regularized
  forward-backward sweep iteration for optimal control problems.
\newblock {\em Submitted}, arXiv: 1912.07028v1, 2019.

\bibitem{MRS13}
Y.~Maday, M.-K. Riahi, and J.~Salomon.
\newblock Parareal in time intermediate targets methods for optimal control
  problems.
\newblock In {\em Control and optimization with {PDE} constraints}, volume 164
  of {\em Internat. Ser. Numer. Math.}, pages 79--92. Birkh\"{a}user/Springer
  Basel AG, Basel, 2013.

\bibitem{SS16}
J.~M. Sanz-Serna.
\newblock Symplectic {R}unge-{K}utta schemes for adjoint equations, automatic
  differentiation, optimal control, and more.
\newblock {\em SIAM Rev.}, 58(1):3--33, 2016.

\bibitem{SSV98}
B.~Sommeijer, L.~Shampine, and J.~Verwer.
\newblock {RKC}: an explicit solver for parabolic {PDEs}.
\newblock {\em J. Comput. Appl. Math.}, 88:316--326, 1998.

\bibitem{SV80}
B.~P. Sommeijer and J.~G. Verwer.
\newblock {\em A performance evaluation of a class of
  {R}unge-{K}utta-{C}hebyshev methods for solving semidiscrete parabolic
  differential equations}.
\newblock Afdeling Numerieke Wiskunde [Department of Numerical Mathematics],
  91. Mathematisch Centrum, Amsterdam, 1980.

\bibitem{HoS80}
P.~J. van~der Houwen and B.~P. Sommeijer.
\newblock On the internal stability of explicit, {$m$}-stage {R}unge-{K}utta
  methods for large {$m$}-values.
\newblock {\em Z. Angew. Math. Mech.}, 60(10):479--485, 1980.

\bibitem{Ver96b}
J.~Verwer.
\newblock Explicit {R}unge-{K}utta methods for parabolic partial differential
  equations.
\newblock {\em Special issue of Appl. Num. Math.}, 22:359--379, 1996.

\bibitem{Wal07}
A.~Walther.
\newblock Automatic differentiation of explicit {R}unge-{K}utta methods for
  optimal control.
\newblock {\em Comput. Optim. Appl.}, 36(1):83--108, 2007.

\bibitem{Zb11}
C.~J. Zbinden.
\newblock Partitioned {R}unge-{K}utta-{C}hebyshev methods for
  diffusion-advection-reaction problems.
\newblock {\em SIAM J. Sci. Comput.}, 33(4):1707--1725, 2011.

\end{thebibliography}
\end{document}